\documentclass[12pt,reqno]{amsart}

\usepackage{amssymb,amsmath,amsthm,amsrefs,tikz}
\usetikzlibrary{matrix,arrows}
\allowdisplaybreaks[2]
\usepackage[margin=1in]{geometry}

\usepackage{calc}
\usepackage{array}
\newenvironment{matrixA}
	{\left[\begin{array}{
		>{\centering\arraybackslash$} p{2.5ex} <{$}  
		>{\centering\arraybackslash$} p{\widthof{$0$}} <{$}  
		>{\centering\arraybackslash$} p{\widthof{$-q^{-1}-q$}} <{$}  
		>{\centering\arraybackslash$} p{\widthof{$0$}} <{$}  }
	}{\end{array}\right]}
\newenvironment{matrixB}
	{\left[\begin{array}{
		>{\centering\arraybackslash$} p{2ex} <{$}  
		>{\centering\arraybackslash$} p{\widthof{$1+q^{2}$}} <{$}  
		>{\centering\arraybackslash$} p{\widthof{$q^2$}} <{$}  
		>{\centering\arraybackslash$} p{\widthof{$1.$}} <{$}  }
	}{\end{array}\right]}
\newenvironment{matrixC}
	{\left[\begin{array}{
		>{\centering\arraybackslash$} p{2.5ex} <{$}  
		>{\centering\arraybackslash$} p{6.5ex} <{$}  
		>{\centering\arraybackslash$} p{3ex} <{$}  
		>{\centering\arraybackslash$} p{3ex} <{$}  }
	}{\end{array}\right]}
\newenvironment{matrixD}
	{\left[\begin{array}{
		>{\centering\arraybackslash$} p{4.5ex} <{$}  
		>{\centering\arraybackslash$} p{6ex} <{$}  
		>{\centering\arraybackslash$} p{2.5ex} <{$}  
		>{\centering\arraybackslash$} p{3ex} <{$}  }
	}{\end{array}\right]}
\newenvironment{matrixE}
	{\left[\begin{array}{
		>{\centering\arraybackslash$} p{2.5ex} <{$}  
		>{\centering\arraybackslash$} p{\widthof{$1-q^{2}$}} <{$}  
		>{\centering\arraybackslash$} p{\widthof{$q-q^{2}$}} <{$}  }
	}{\end{array}\right]}
\newenvironment{matrixF}
	{\left[\begin{array}{
		>{\centering\arraybackslash$} p{\widthof{$q+q^{1}$}} <{$}  
		>{\centering\arraybackslash$} p{2ex} <{$}  
		>{\centering\arraybackslash$} p{2.5ex} <{$}  }
	}{\end{array}\right]}
\newenvironment{matrixG}
	{\left[\begin{array}{
		>{\centering\arraybackslash$} p{2.5ex} <{$}  
		>{\centering\arraybackslash$} p{2ex} <{$}  
		>{\centering\arraybackslash$} p{5.5ex} <{$}  }
	}{\end{array}\right]}
\newenvironment{matrixH}
	{\left[\begin{array}{
		>{\centering\arraybackslash$} p{6.5ex} <{$}  
		>{\centering\arraybackslash$} p{5ex} <{$}  
		>{\centering\arraybackslash$} p{2.5ex} <{$}  }
	}{\end{array}\right]}

\numberwithin{equation}{section}

\newtheorem{theorem}{Theorem}[section]
\newtheorem{proposition}[theorem]{Proposition}
\newtheorem{conjecture}[theorem]{Conjecture}
\newtheorem{corollary}[theorem]{Corollary}
\newtheorem{lemma}[theorem]{Lemma}

\newtheorem{maintheorem}[theorem]{Main Theorem}
\newtheorem{mainconjecture}[theorem]{Main Conjecture}

\theoremstyle{definition}

\newtheorem{remark}[theorem]{Remark}
\newtheorem{example}[theorem]{Example}
\newtheorem{definition}[theorem]{Definition}

\newtheorem{problem}[theorem]{Problem}
\newcommand{\qbinom}[3][q]{\genfrac[]{0pt}0{#2}{#3}_{#1}}

\newcommand{\Mat}{\mathrm{Mat}}
\newcommand{\op}{\mathrm{op}}
\newcommand{\one}{\mathbf 1}

\newcommand{\counit}{\varepsilon}
\renewcommand{\eqref}[1]{{\rm (\ref{#1})}}

\def\ZZ{\mathbb{Z}}
\def\CC{\mathbb{C}}

\def\Hom{\operatorname{Hom}}
\def\End{\operatorname{End}}
\def\Ext{\operatorname{Ext}}

\def\id{\operatorname{id}}

\newcommand{\lr}[1]{{\langle #1\rangle}}

\newcommand{\bfv}{\mathbf{v}}

\newcommand{\cA}{\mathcal{A}}

\newcommand{\cC}{\mathcal{C}}

\newcommand{\cH}{\mathcal{H}}

\newcommand{\cR}{\mathcal{R}}

\newcommand{\rD}{D}
\newcommand{\rU}{U}
\newcommand{\rSL}{\mathrm{SL}}
\newcommand{\rGL}{\mathrm{GL}}

\newcommand{\fg}{\mathfrak{g}}

\newcommand{\ft}{\mathfrak{t}}
\newcommand{\fsl}{\mathfrak{sl}}
\newcommand{\fgl}{\mathfrak{gl}}

\newcommand{\actleft}{\vartriangleright}
\newcommand{\actright}{\vartriangleleft}
\newcommand{\blactleft}{\blacktriangleright}
\newcommand{\blactright}{\blacktriangleleft}

\newcommand{\dualdouble}{ \rU_q(\fsl_2) \otimes \CC_q[\rSL_2] }
\newcommand{\HCpp}{ {^+}(H\otimes C){^+} }
\newcommand{\Hpp}{ {^+}\mathcal{H}^+ }
\newcommand{\vd}{\dot v}
\newcommand{\vdd}{\ddot v}
\newcommand{\Cplus}{ {^+C} }
\DeclareMathOperator{\Span}{span}

\begin{document}

\title{Semisimplicity of certain representation categories}

\author{John E.~Foster}
\address{\noindent Department of Mathematics, Walla Walla University, College Place, WA 99324, USA} \email{john.foster@wallawalla.edu}

\maketitle
\tableofcontents
\section{Introduction}

Often in studying representation theory we find that certain representation categories are semisimple.  This is an illuminating property, since in such categories it suffices to study the simple objects.  The research presented in this paper is motivated by the desire to expand our available tools for proving semisimplicity, with a special focus on representations of the quantum double.

A familiar result in this direction is that representations of semisimple Lie algebras are semisimple \cite{MR0323842}*{28}.  Pivotal to the proof of this theorem is the existence of a \emph{Casimir} element with certain properties.  The proof generalizes nicely to quantized enveloping algebras \cite{MR953821}*{587--589}.  In Section \ref{section:casimir} we further generalize the proof to a Hopf algebra $H$, which covers both cases.

\begin{theorem}\label{theorem:weyl_for_Hopf_algebras}
Let $H$ be a Hopf algebra and let $\cC$ be an Abelian category of finite-dimensional $H$-modules which is closed under extension.  Suppose that $H$ has a Casimir element which acts by 0 on a simple module $V$ if and only if $V$ is the trivial module.  If the extension of the trivial module by itself is 0, then $\cC$ is semisimple.
\end{theorem}

In cases where the center of $H$ is not well understood and no such Casimir element is known, we choose to pursue a different approach.  Matrix coefficients of representations of Lie groups were first described by \'Elie Cartan, and they were used by Fritz Peter and Hermann Weyl in the 1920's to decompose representations of compact topological groups in their famous Peter-Weyl theorem.  Israel Gelfand continued using matrix coefficients of representations to bring new insight to several classical problems.  Their work is the inspiration for our approach.

In Section \ref{section:correspondence} we describe a correspondence between algebra representation categories and sub-bimodules of the dual of the algebra.  We show that if $A$ is an algebra and $V$ is a finite-dimensional left $A$-module, then there is a bimodule morphism
\begin{equation}\label{eqn:matrix_coefficients}
	\beta_V\colon V \otimes V^* \to A^*  
	\qquad \text{given by} \qquad 
	\beta_V(v \otimes \zeta)(a) = \zeta(a \actleft v).
\end{equation}
We thus view $A^*$ as the best place to look for $A$-modules, and we examine some other properties of this correspondence.

In Section \ref{section:PW_theorem} we establish a Peter-Weyl-type theorem that makes use of this correspondence to prove semisimplicity of a category.

\begin{theorem}\label{thm:PWdual}
Let $B$ be a bialgebra and $\cC$ be an Abelian category of finite-dimensional $B$-modules.  Then $\cC$ is semisimple if and only if the image $B_\cC^*$ of $\cC$ under the correspondence \eqref{eqn:matrix_coefficients} has Peter-Weyl decomposition
\begin{equation*}
	B_\mathcal{C}^* = \bigoplus_{V} \beta_V(V\otimes V^*)
\end{equation*}
as an internal direct sum over all isomophism classes in $\cC$.
\end{theorem}

Our goal is to make use of these ideas to establish semisimplicity in a new situation.  It is well-known that if $H$ is a finite-dimensional Hopf algebra and if representations of $H$ and $H^*$ are semisimple, then so are representations of the quantum double $\rD(H)$ \cite{MR1243637}*{193}.  This is not necessarily true when $H$ is infinite-dimensional; for example, not all finite-dimensional representations of $\rD(\rU(\fsl_2))$ are semisimple.  Neither are all finite-dimensional representations of $\rU_q(\fsl_2)$, nor of its double $\rD(\rU_q(\fsl_2))$, when $q$ is specialized to a root of unity.  However, in the author's conversations with Victor Ostrik, the following conjecture was made for generic $q$.  It appears to be an open problem, and very difficult, even when $\fg=\fsl_2$.

\begin{mainconjecture}\label{conjecture:victor}
Let $\fg$ be a semisimple Lie algebra.  Then every finite-dimensional representation of $\rD(\rU_q(\fg))$ is semisimple.
\end{mainconjecture}

In Section \ref{section:Vplusminus} we further conjecture what the simple $\rD(\rU_q(\fg))$-modules are, as we now describe.  Suppose that $H$ is a Hopf algebra with invertible quasi-triangular structure $R$.  Given a left $H$-module $V$, we can construct two left $\rD(H)$-modules $V^+$ and $V^-$ using $R$ and $R^{-1}$, respectively.

\begin{lemma}\label{lemma:DH_simples}
Suppose that all left $H$-modules are semisimple, that $V\otimes V^*$ is semisimple for any simple $H$-module $V$, and that $V^+\ncong V^-$ if $V$ is non-trivial.  If $U$ and $V$ are simple $H$-modules, then the $\rD(H)$-module $U^+\otimes V^-$ is simple.
\end{lemma}

It is unclear whether we have accounted for all simple $\rD(H)$-modules: Although the additive span of such $U^+\otimes V^-$ is closed under tensor multiplication, we do not know whether it is closed under extension.  In the author's conversations with Victor Ostrik, the following conjecture was made.

\begin{conjecture}\label{conjecture:victor2}
Let $\fg$ be a semisimple Lie algebra.  The simple $\rD(\rU_q(\fg))$-modules are, up to isomorphism, the modules $V_\lambda^+\otimes V_\mu^-\otimes U_0$ where $\lambda$ and $\mu$ are dominant integral $\fg$-weights and
$U_0$ belongs to the (finite) set of one-dimensional $\rD(\rU_q(\fg))$-modules.
\end{conjecture}

Because little is known about the center of $\rD(\rU_q(\fg))$, we have little hope of applying Theorem \ref{theorem:weyl_for_Hopf_algebras}.  Our main conjecture appears to be very difficult, so even a little progress would be quite helpful.  In this paper we consider the semisimple Lie algebra $\fg=\fsl_2$ and attempt to apply Theorem \ref{thm:PWdual}.

\begin{remark}\label{remark:advantage_of_bimodule}
One advantage of this strategy is that $B^*$ is not just a left $B$-module, but also a right $B$-module.  In the case of $\rD(H)$, the right action binds together the multiplicities of each left $\rD(H)$-module into a single $\rD(H)$-bimodule without multiplicity.
\end{remark}

Now when $H$ is infinite-dimensional, the structure of $\rD(H)^*$ is very complicated.  If $C$ is the finite dual of $H$, then $\rD(H)$ is the coalgebra $C\otimes H$ with the necessary algebra structure to make it into a Hopf algebra, as we explain in Section \ref{section:the_quantum_double}.  It follows that the Hopf algebra $\rD(H)^*$ contains $H\otimes C$ as a subalgebra with the tensor algebra structure.  The action of $\rD(H)$ on $H\otimes C$ is determined completely by the Hopf pairing of $H$ and $C$.  Now $H\otimes C=\rD(H)^*$ when $H$ is finite-dimensional, but not when $H$ is infinite-dimensional.  However, we show in Section \ref{section:double_actions} that $H\otimes C$ is still a surprisingly interesting object in the infinite-dimensional case.

\begin{theorem}
The subalgebra $H\otimes C\subset \rD(H)^*$ is a sub-bimodule.
\end{theorem}

In Section \ref{section:correspondence} we define $(H\otimes C)_f$ to be the sub-bimodule of elements of $H\otimes C$ which generate finite-dimensional $\rD(H)$-bimodules, and we note that this space is also an algebra.
We remark here that $(H\otimes C)_f$ is not a new object; in fact, for any Hopf algebra $H$, the subalgebra $(H\otimes C)_f \subset \rD(H)^*$ 
is the intersection of $H\otimes C$ with the Sweedler dual $\rD(H)^\circ$.

In Section \ref{section:main_results} we focus our attention on the Hopf algebra $\rU_q(\fsl_2)$.  Now for finite-dimensional $H$, an element of $H\otimes C$ is locally-finite under the actions of $\rD(H)$ if and only if it is locally-finite under the actions of $H$.  It is very helpful that the same holds for $H=\rU_q(\fsl_2)$.

\begin{proposition}
An element of $\dualdouble \subset \rD(\rU_q(\fsl_2))^*$ is locally-finite under the actions of $\rD(\rU_q(\fsl_2))$ if and only if it is locally-finite under the actions of $\rU_q(\fsl_2)$.
\end{proposition}

Just as highest-weight vectors are key to the study of $\rU_q(\fsl_2)$-modules, so too are highest-weight bivectors (those which are highest-weight on both the left and the right) key to achieving these results.  We use Remark \ref{remark:advantage_of_bimodule} extensively.  We show that one particular $\rD(\rU_q(\fsl_2))$-bimodule, denoted $\cH_{1,1}$, has a set $\{\bfv_1,\bfv_2,\bfv_3,\bfv_4\}$ of four canonical highest-weight bivectors which play a very important role in the representation theory of $\rD(\rU_q(\fsl_2))$.

\begin{proposition}
The subalgebra of highest-weight bivectors in $\dualdouble$ is locally-finite and is generated as an algebra by $\{\bfv_1,\bfv_2,\bfv_3,\bfv_4\}$.
\end{proposition}

This result helps us to prove our main theorem.  We write $\cH_{\lambda,\mu}$ to denote the image of $\beta_{V_\lambda^+\otimes V_\mu^-}$ from \eqref{eqn:matrix_coefficients}.

\begin{maintheorem}
As a $\rD(\rU_q(\fsl_2))$-bimodule,
\begin{equation*}
	\big(\rU_q(\fsl_2) \otimes \CC_q[\rSL_2]\big)_f
	=
	\bigoplus_{\genfrac{}{}{0pt}{}{\lambda,\mu\geq 0}{\lambda-\mu \in 2\ZZ}} \cH_{\lambda,\mu}
\end{equation*}
and this is a Peter-Weyl decomposition.
\end{maintheorem}

By Theorem \ref{thm:PWdual} this proves semisimplicity of a substantial subcategory of finite-dimensional $\rD(\rU_q(\fsl_2))$-modules.

At the end of Section \ref{section:main_results} we give a presentation and basis for the algebra of highest-weight bivectors (This algebra was quite surprising to the author, suggesting that the algebra of all locally-finite bivectors merits further investigation).  We also generalize our main results to conjectures for other semisimple Lie algebras.\bigskip

\noindent\textbf{Acknowledgements.} The author is tremendously indebted to his advisor, Arkady Berenstein, for his guidance, encouragement, and enthusiasm.  The author also thanks Victor Ostrik for enlightening discussions.

\section{Background}

\subsection{Hopf algebras, modules, and comodules}

The material in this section is well known and can be found in many sources, such as Chapter 1 of \cite{MR1381692}.

\begin{definition}
An \emph{algebra} $A$ over a field $k$ is a $k$-vector space with a linear \emph{multiplication} map $\mu\colon A\otimes A\to A$ satisfying
$ \mu(\mu(a, b), c) = \mu(a, \mu(b, c)) $
and a linear \emph{unit} map $\eta\colon k\to A$ satisfying $\mu(\eta(1), a)=\mu(a, \eta(1)) = a$, for all $a,b,c\in A$.
\end{definition}

In more familiar notation, these conditions are written $(ab)c=a(bc)$, which we thus write $abc$, and $1a=a1=a$.  We will usually use this more familiar notation, but the advantage of the given definition is that its conditions can be expressed using commutative diagrams (See Figure \ref{fig:multiplication_and_unit}).
\begin{figure}[ht]
\begin{center}
\begin{tikzpicture}[description/.style={fill=white,inner sep=2pt}]
	\matrix (m) [matrix of math nodes, row sep=3em, column sep=2.5em, text height=1.5ex, text depth=0.25ex]
	{ A \otimes A \otimes A & A \otimes A \\
	  A \otimes A & A \\ };
		\path[->,font=\scriptsize]
	(m-1-1)	edge node[auto] {$ 1 \otimes \mu $} (m-1-2)
	(m-1-2)	edge node[auto] {$ \mu $} (m-2-2)
	(m-2-1)	edge node[auto] {$ \mu $} (m-2-2);
		\path[<-,font=\scriptsize]
	(m-2-1)	edge node[auto] {$ \mu \otimes 1 $} (m-1-1);
\end{tikzpicture}
\qquad
\begin{tikzpicture}[description/.style={fill=white,inner sep=2pt}]
	\matrix (m) [matrix of math nodes, row sep=3em, column sep=2.5em, text height=1.5ex, text depth=0.25ex]
	{ A \otimes A & k \otimes A\\
	  A \otimes k & A \\ };
		\path[<-,font=\scriptsize]
	(m-1-1)	edge node[auto] {$ \eta \otimes 1 $} (m-1-2);
		\path[->,font=\scriptsize]
	(m-2-1)	edge node[auto] {$ 1 \otimes \eta $} (m-1-1)
	(m-1-1)	edge node[auto] {$ \mu $} (m-2-2)
	(m-1-2)	edge node[auto] {$ = $} (m-2-2)
	(m-2-1)	edge node[auto] {$ = $} (m-2-2);
\end{tikzpicture}
\end{center}
\caption{Conditions on the multiplication and unit of an algebra}
\label{fig:multiplication_and_unit}
\end{figure}
This makes it easy to define a dual notion, that of a \emph{coalgebra} over $k$.

\begin{definition}
A \emph{coalgebra} $C$ over a field $k$ is a $k$-vector space with a linear \emph{comultiplication} map $\Delta\colon C\to C\otimes C$ satisfying
$ (\Delta\otimes 1)\circ\Delta=(1\otimes\Delta)\circ\Delta $
and a linear \emph{counit} map $\counit\colon C\to k$ satisfying $(1\otimes\counit)(\Delta c)=(\counit\otimes 1)(\Delta c)=c$ for all $c\in C$.
\end{definition}

\begin{figure}[ht]
\begin{center}
\begin{tikzpicture}[description/.style={fill=white,inner sep=2pt}]
	\matrix (m) [matrix of math nodes, row sep=3em, column sep=2.5em, text height=1.5ex, text depth=0.25ex]
	{ C \otimes C \otimes C & C \otimes C \\
	  C \otimes C & C \\ };
		\path[<-,font=\scriptsize]
	(m-1-1)	edge node[auto] {$ 1 \otimes \Delta $} (m-1-2)
	(m-1-2)	edge node[auto] {$ \Delta $} (m-2-2)
	(m-2-1)	edge node[auto] {$ \Delta $} (m-2-2);
		\path[->,font=\scriptsize]
	(m-2-1)	edge node[auto] {$ \Delta \otimes 1 $} (m-1-1);
\end{tikzpicture}
\qquad
\begin{tikzpicture}[description/.style={fill=white,inner sep=2pt}]
	\matrix (m) [matrix of math nodes, row sep=3em, column sep=2.5em, text height=1.5ex, text depth=0.25ex]
	{ C \otimes C & k \otimes C\\
	  C \otimes k & C \\ };
		\path[->,font=\scriptsize]
	(m-1-1)	edge node[auto] {$ \counit \otimes 1 $} (m-1-2);
		\path[<-,font=\scriptsize]
	(m-2-1)	edge node[auto] {$ 1 \otimes \counit $} (m-1-1)
	(m-1-1)	edge node[auto] {$ \Delta $} (m-2-2)
	(m-1-2)	edge node[auto] {$ = $} (m-2-2)
	(m-2-1)	edge node[auto] {$ = $} (m-2-2);
\end{tikzpicture}
\end{center}
\caption{Conditions on the comultiplication and counit of a coalgebra}
\label{fig:comultiplication_and_counit}
\end{figure}

We use Sweedler's notation $\Delta(c)=c_{(1)}\otimes c_{(2)}$, where it is understood that this is often a \emph{sum} of elementary tensors.  In this notation, the counit condition is $c_{(1)}\counit(c_{(2)})=\counit(c_{(1)})c_{(2)}=c$, and the coassociativity condition is
\[ (c_{(1)})_{(1)} \otimes (c_{(1)})_{(2)} \otimes c_{(2)} = c_{(1)} \otimes (c_{(2)})_{(1)} \otimes (c_{(2)})_{(2)}. \]
This coproduct may thus be written $\Delta^2 c = c_{(1)} \otimes c_{(2)} \otimes c_{(3)}$.

\begin{definition}
An element $c$ of a coalgebra is called \emph{group-like} if $\Delta c = c\otimes c$.
\end{definition}

\begin{definition}
A \emph{bialgebra} $B$ over a field $k$ is a $k$-vector space that is both an algebra and a coalgebra over $k$ such that $\Delta$ and $\counit$ are algebra homomorphisms.  Equivalently, $\mu$ and $\eta$ are coalgebra homomorphisms.
\end{definition}

\begin{definition}
A \emph{Hopf algebra} $H$ is a bialgebra with a linear \emph{antipode} map $S\colon H\to H$ such that
$ S(h_{(1)})h_{(2)}=h_{(1)}S(h_{(2)})=\counit(h). $
(See Figure \ref{fig:antipode}.)
\end{definition}

\begin{figure}[ht]
\begin{center}
\begin{tikzpicture}[description/.style={fill=white,inner sep=2pt}]
	\matrix (m) [matrix of math nodes, row sep=3em, column sep=2.5em, text height=1.5ex, text depth=0.25ex]
	{ H & k & H \\
	  H \otimes H & & H \otimes H \\ };
		\path[<-,font=\scriptsize]
	(m-2-1)	edge node[auto] {$ \Delta $} (m-1-1)
	(m-1-3)	edge node[auto] {$ \mu $} (m-2-3);
		\path[->,font=\scriptsize]
	(m-1-1)	edge node[auto] {$ \counit $} (m-1-2)
	(m-1-2)	edge node[auto] {$ \eta $} (m-1-3)
	(m-2-1)	edge node[auto] {$ S \otimes 1 , 1 \otimes S $} (m-2-3);
\end{tikzpicture}
\end{center}
\caption{Condition on the antipode of a Hopf algebra}
\label{fig:antipode}
\end{figure}

\begin{definition}
Let $A$ be an algebra over $k$.  A \emph{left $A$-module} $V$ is a $k$-vector space with a linear \emph{action} $\actleft\colon A\otimes V\to V$ satisfying $1\actleft v=v$ and $ab\actleft v=a\actleft(b\actleft v)$ for all $a,b\in A$ and $v\in V$. Right actions and modules are defined similarly. An \emph{$A$-bimodule} has left and right actions satisfying $(a\actleft v)\actright b=a\actleft(v\actright b)$ for all $a,b\in A$ and $v\in V$.
\end{definition}

\begin{definition}
Let $C$ be a coalgebra over $k$.  A \emph{left $C$-comodule} $V$ is a $k$-vector space with a linear \emph{coaction} $\delta\colon V\to C\otimes V$ satisfying $(\counit\otimes 1)(\delta v) = v$ and $(1\otimes\delta)\circ\delta = (\Delta\otimes 1)\circ\delta$ for all $v\in V$. Right coactions and comodules are defined similarly.
\end{definition}

We write $\delta(v)=v^{(-1)}\otimes v^{(0)}$ (which may in fact be a sum of elementary tensors).  Then the first condition on $\delta$ can be written $\counit\big(v^{(-1)}\big)v^{(0)}=v$, and the second condition allows us to write $\delta^2 v=v^{(-2)}\otimes v^{(-1)}\otimes v^{(0)}$.

\begin{lemma}\label{lem:dual_bialgebra}
Let $A$ be an algebra.  Then $A^*$ is an $A$-bimodule where the left action is $(a\actleft f)(a')=f(aa')$ and the right action is $(f\actright a)(a')=f(a'a)$.
\end{lemma}

\begin{lemma}\label{lem:tensor_algebra}
The category of algebras over $k$ is monoidal, where $A\otimes B$ has multiplication
\[ (a\otimes b)(a'\otimes b') = aa'\otimes bb' \]
and unit $\eta_A(1)\otimes \eta_B(1)$.
\end{lemma}

\begin{lemma}\label{lem:tensor_coalgebra}
The category of coalgebras over $k$ is monoidal, where $A\otimes B$ has comultiplication
\[ \Delta(a\otimes b) = (a_{(1)}\otimes b_{(1)})\otimes (a_{(2)}\otimes b_{(2)}) \]
and counit $\counit(a\otimes b)=\counit_A(a)\counit_B(b)$.
\end{lemma}

\begin{lemma}\label{lem:tensor_mod}
Let $B$ be a bialgebra.  The category of left $B$-modules is monoidal with 
\[ b \actleft (u\otimes v) = (b_{(1)} \actleft u) \otimes (b_{(2)} \actleft v). \]
\end{lemma}

\begin{lemma}\label{lem:tensor_comod}
Let $B$ be a bialgebra.  The category of left $B$-comodules is monoidal with
\[ \delta (u\otimes v) = u^{(-1)}v^{(-1)}\otimes u \otimes v. \]
\end{lemma}

\subsection{Examples of bialgebras and Hopf algebras}

The construction we use here for $\rU_q(\fg)$ is detailed in \cite{MR1381692}*{95--97}.
Let $\fg$ be a complex simple Lie algebra, $\ft$ be a Cartan subalgebra, and $\ft^*$ be its dual linear space.  Let $\alpha_i\in\ft^*$ be a system of positive simple roots.  If $(\;,\,)$ is the symmetric bilinear form on $\ft^*$ derrived from the inverse of the Killing form, and $\check{\alpha}_i=2\alpha_i/(\alpha_i,\alpha_i)$ are the coroots, then $a_{ij}=(\check{\alpha}_i,\alpha_j)$ is the Cartan matrix.  Define $d_i=(\alpha_i,\alpha_i)/2$, which is always an integer.  Then $\rU_q(\fg)$ can be defined over the field $\CC(q)$ with generators $\{K_i^{\pm 1},E_i,F_i\}$ and with $q_i=q^{d_i}$.

\begin{definition}\label{def:quantum_enveloping_algebra}
Let $\fg$ be a complex simple Lie algebra with the structure described above.  Using the notation $[n]_q=\frac{q^n-q^{-n}}{q-q^{-1}}$ and $[n]_q!=[n]_q[n-1]_q\dotsb[1]_q$ with $\textstyle\big[\genfrac{}{}{0pt}{}{m}{n}\big]_q=\frac{[m]_q!}{[n]_q![m-n]_q!}$, we define $\rU_q(\fg)$ to be the Hopf algebra generated by $\{K_i^{\pm 1},E_i,F_i\}$ with relations $[K_i,K_j]=0$,
\begin{gather*}
	K_iE_jK_i^{-1}=q^{a_{ij}}E_j,\quad K_iF_jK_i^{-1}=q^{-a_{ij}}F_j,\quad[E_i,F_j]=\delta_{ij}\frac{K_i-K_i^{-1}}{q_i-q_i^{-1}},\\
	\sum_{k=0}^{1-a_{ij}}(-1)^k\qbinom[q_i]{1-a_{ij}}{k}E_i^{1-a_{ij}-k}E_jE_i^k=0\quad \text{for all $i\neq j$,}\\
	\sum_{k=0}^{1-a_{ij}}(-1)^k\qbinom[q_i]{1-a_{ij}}{k}F_i^{1-a_{ij}-k}F_jF_i^k=0\quad \text{for all $i\neq j$,}
\end{gather*}
and comultiplication, counit, and antipode maps
\begin{gather*}
	\Delta K_i=K_i\otimes K_i,\quad \Delta E_i=E_i\otimes K_i+1\otimes E_i,\quad \Delta F_i=F_i\otimes 1+K_i^{-1}\otimes F_i,\\
	\counit(K_i)=1,\quad \counit(E_i)=\counit(F_i)=0,\\
	SK_i=K_i^{-1},\quad SE_i=-E_iK_i^{-1},\quad SF_i=-K_iF_i.
\end{gather*}
\end{definition}

\begin{remark}
For $\rU_q(\fsl_2)$ there is a Casimir element
\[ \Delta = EF + \frac{q^{-1}K+qK^{-1}}{(q-q^{-1})^2} = FE + \frac{qK+q^{-1}K^{-1}}{(q-q^{-1})^2} \]
which belongs to the center of $\rU_q(\fsl_2)$.
\end{remark}

If $V$ is a left $\rU_q(\fg)$-module, a nonzero vector $v\in V$ is said to be \emph{highest-weight} if $E_i\actleft v=0$ for all $i$, and the subspace of highest-weight vectors is denoted $^+V$.  We define $V^+$ similarly for right $\rU_q(\fg)$-modules.  The following result is a consequence of the comultiplication of $E_i$ and of Lemma \ref{lem:tensor_mod}.

\begin{lemma}\label{lem:hw_subalgebra}
If $V$ is a $\rU_q(\fg)$-module, then ${^+}V$ (resp.\ $V^+$) is a subalgebra.
\end{lemma}

We say that a module $V$ is \emph{locally finite} if every $v\in V$ generates a finite-dimensional submodule.

\begin{lemma}\label{lem:hw_generate_locally_finite_modules}
If a $\rU_q(\fg)$-module $V$ is locally finite, then $^+V$ (resp.\ $V^+$) generates $V$.
\end{lemma}

Quantum matrices are another class of bialgebras. The construction we use is detailed in \cite{MR1016381}*{5--6}.  Let the coordinate functions on the space of $m\times n$ matrices be denoted $X_{ij}$, $1\leq i\leq m$, $1\leq j\leq n$.  The tensor algebra $T(\Mat_{m\times n})$ generated by $\{X_{ij}\}$ is a bialgebra with comultiplication and counit given by
\begin{equation}\label{eqn:matrix_commult_and_counit}
	\Delta(X_{ij})=\textstyle\sum\nolimits_k X_{ik}\otimes X_{kj}\quad\text{and}\quad\counit(X_{ij})=\delta_{i,j}.
\end{equation}

Consider the free algebra $V$ on $n$ generators $e_1,\dotsc,e_n$.  This algebra is a comodule over the bialgebra $T(\Mat_{n\times n})$ with coaction
\begin{equation}\label{eq:coaction}
\delta(e_i) = \sum_{j=1}^n X_{ij} \otimes e_j.
\end{equation}

The bialgebra of quantum matrices $\CC_q[\Mat_{n\times n}]$ is a quotient of the tensor bialgebra $T(\Mat_{n\times n})$ such that both the quantum symmetric algebra and the quantum exterior algebra are comodules.  More precisely, we define the symmetric algebra $S_q(V)$ and the exterior algebra $\Lambda_q(V)$ by
\begin{align*}
	S_q(V) &= T(V)/(e_je_i-qe_ie_j\mid i<j)\\\text{and}\quad
	\Lambda_q(V) &= T(V)/(e_j\wedge e_i+q^{-1} e_i\wedge e_j\mid i\leq j).
\end{align*}

The following result is well-known (See \cite{MR1016381}, for example).
\begin{theorem}
There exists a quadratic bi-ideal $I$ in the bialgebra $T(\Mat_{n\times n})$ such that \eqref{eq:coaction} extends to algebra homomorphisms 
\begin{subequations}\label{eq:coactions_for_quantum_matrices}
\begin{align}
	S_q(V)&\to \big(T(\Mat_{n\times n})/I\big)\otimes S_q(V)\\\text{and}\quad
	\Lambda_q(V)&\to \big(T(\Mat_{n\times n})/I\big)\otimes \Lambda_q(V).
\end{align}
\end{subequations}
\end{theorem}

We define $ \CC_q[\Mat_{n\times n}] $ to be the quotient of $T(\Mat_{n\times n}) $ by the minimal such $I$.  Since the coactions \eqref{eq:coactions_for_quantum_matrices} are homogeneous, we may restrict the latter coaction to the top power $\big(\Lambda_q(V)\big)^n$, which is 1-dimensional.  The image of this restriction is spanned by a central, group-like element of $\CC_q[\Mat_{n\times n}]$ which we call the quantum determinant $\det_q(\Mat_{n\times n})$ of $\CC_q[\Mat_{n\times n}]$.  If the quantum determinant is made to be invertible, this will produce an antipode.

\begin{definition}\label{def:quantum_matrices_gln}
The Hopf algebra $\CC_q[\rGL_n]$ is the localization of the quantum matrix bialgebra $\CC_q[\Mat_{n\times n}]$ at the quantum determinant $\det_q(\Mat_{n\times n})$.
\end{definition}

\begin{definition}
The Hopf algebra $\CC_q[\rSL_n]$ is the quotient of the quantum matrix bialgebra $\CC_q[\Mat_{n\times n}]$ by imposing $\det_q(\Mat_{n\times n}) =1$.
\end{definition}

\begin{example}\label{ex:coordinate_algebra}
We give a presentation of $\CC_q[\rSL_2]$ with the four generators $a=X_{11}$, $b=X_{12}$, $c=X_{21}$, and $d=X_{22}$.  The algebra relations are
\begin{gather*}
	ba = qab,\quad ca = qac,\quad db = qbd,\quad dc = qcd,\\
	cb = bc,\quad ad-q^{-1}bc = da-qbc = 1.
\end{gather*}
The comultiplication and counit are given in \eqref{eqn:matrix_commult_and_counit}, and the antipode $S$ is
\begin{equation*}
	S\begin{pmatrix}a&b\\c&d\end{pmatrix}
	=\begin{pmatrix}d&-qb\\-q^{-1}c&a\end{pmatrix}
\end{equation*}
which is read entry-wise, so $S(b)=-qb$, for example.
\end{example}

\subsection{Hopf pairings and actions}\label{section:hopf_pairing}

\begin{definition}
Let $H$ and $C$ be Hopf algebras over a field $k$.  A \emph{Hopf pairing} of $H$ and $C$ is a map $\phi\colon C\otimes H\to k$ such that
\begin{align*}
		\phi(c,hh')&=\phi(c_{(1)},h)\phi(c_{(2)},h'),\quad
		\phi(cc',h)=\phi(c,h_{(1)})\phi(c',h_{(2)}),
\end{align*}
and $\phi(1,h)=\counit(h)$, $\phi(c,1)=\counit(c)$, and $\phi(Sc,h)=\phi(c,Sh)$.
\end{definition}

We can now define actions of dually paired Hopf algebras on each other.
\begin{proposition}\label{prop:pairing_action}
Let $H$ and $C$ be Hopf algebras over a field $k$, and let $\phi\colon C\otimes H\to k$ be a Hopf pairing.  Then
\begin{equation*}
	h\actleft c=c_{(1)}\phi(c_{(2)},h)
	\qquad\text{and}\qquad
	c\actright h=c_{(2)}\phi(c_{(1)},h)
\end{equation*}
are left and right actions of $H$ on $C$, respectively, and
\begin{equation*}
	c\actleft h=h_{(1)}\phi(c,h_{(2)})
	\qquad\text{and}\qquad
	h\actright c=h_{(2)}\phi(c,h_{(1)})
\end{equation*}
are left and right actions of $C$ on $H$, respectively.  Furthermore, these actions and the pairing $\phi$ satisfy the relations
\begin{align*}
	\phi(c,hh')&=\phi(c\actright h,h')\quad\text{and}\quad
	\phi(cc',h)=\phi(c,c'\actleft h).
\end{align*}
\end{proposition}

If we know the actions, then we can reconstruct the pairing, as demonstrated in the following example.

\begin{example}\label{ex:pairing}
There is a left action of $\rU_q(\fgl_n)$ on $\CC_q[\rGL_n]$ given by
\begin{align*}
	E_i\actleft X_{k\ell} &= \delta_{i+1,\ell}X_{k,\ell-1}\\
	F_i\actleft X_{k\ell} &= \delta_{i,\ell}X_{k,\ell+1}\\
	K_i\actleft X_{k\ell} &= (q\delta_{i,\ell}+q^{-1}\delta_{i+1,\ell})X_{k\ell}
\end{align*}
We reconstruct the Hopf pairing
\begin{align*}
	\phi( E_i,X_{k\ell} ) 
		&= \phi( 1,E_i\actleft X_{k\ell} )
		=\counit(\delta_{i+1,\ell}X_{ki})=\delta_{i,k,\ell-1}\\
	\phi( F_i,X_{k\ell} )
		&= \phi( 1,F_i\actleft X_{k\ell} )
		=\counit(\delta_{i,\ell}X_{k,\ell+1})=\delta_{i,k-1,\ell}\\
	\phi( K_i,X_{k\ell} )
		&= \counit( q\delta_{i,\ell}X_{k\ell}+q^{-1}\delta_{i+1,\ell}X_{k\ell} )
		=q\delta_{i,k,\ell}+q^{-1}\delta_{i+1,k,\ell}
\end{align*}
where $\delta_{i,j,k}=1$ if $i=j=k$, and $\delta_{i,j,k}=0$ otherwise.  
The right action is then
\begin{align*}
	\delta_{i,k}X_{k+1,\ell} &= X_{k\ell}\actright E_i\\
	\delta_{i,k-1}X_{k-1,\ell} &= X_{k\ell}\actright F_i\\
	q^{-1}\delta_{i+1,k}X_{k\ell}+q\delta_{i,k}X_{k\ell} &= X_{k\ell}\actright K_i
\end{align*}
We can also use the pairing to compute that
\begin{align*}
	X_{k\ell}\actleft E_i &= (q\delta_{i,k,\ell}+q^{-1}\delta_{i-1,k,\ell})E_i+\delta_{i,k,\ell-1}\\
	X_{k\ell}\actleft F_i &= \delta_{k,\ell}F_i+\delta_{i,k-1,\ell}K_i^{-1}\\
	X_{k\ell}\actleft K_i &= (q\delta_{i,k,\ell}+q^{-1}\delta_{i+1,k,\ell})K_i
\end{align*}
and
\begin{align*}
	\delta_{i,k,\ell-1}K_i+\delta_{k,\ell}E_i &= E_i\actright X_{k\ell}\\
	\delta_{i,k-1,\ell}+(q^{-1}\delta_{i,k,\ell}+q\delta_{i+1,k,\ell})F_i &= F_i\actright X_{k\ell}\\
	(q\delta_{i,k,\ell}+q^{-1}\delta_{i+1,k,\ell})K_i &= K_i\actright X_{k\ell}
\end{align*}
\end{example}
In Example \ref{ex:duqgln_cross_relations} and Section \ref{section:duqsln_actions} we will use these actions to construct $\rD(\rU_q(\fsl_n))$ and to determine the action of $\rD(\rU_q(\fsl_n))$ on part of its dual.

\subsection{Quasi-triangular structures and braidings}

As a reference for the material in this section, see Chapter 2 of \cite{MR1381692}.

\begin{definition}
Let $B$ be a bialgebra.  A quasi-triangular structure on $B$ is an element $R\in B\otimes B$, written $R^{(1)}\otimes R^{(2)}$ though it may be a sum, which is invertible and satisfies
\begin{align*}
	(\Delta\otimes 1)R&=R_{13}R_{23},\quad
	(1\otimes\Delta)R=R_{13}R_{12},\quad\text{and}\quad
	\tau(\Delta b)=R(\Delta b)R^{-1}
\end{align*}
for all $b\in B$, where $\tau(\Delta b)=b_{(2)}\otimes b_{(1)}$ and where $R_{12}=R^{(1)}\otimes R^{(2)}\otimes 1$, $R_{13}=R^{(1)}\otimes 1\otimes R^{(2)}$, and $R_{23}=1\otimes R^{(1)}\otimes R^{(2)}$.
\end{definition}

\begin{example}\label{ex:r_for_uqsl2}
The Hopf algebra $\rU_q(\fsl_2)$ has quasi-triangular structure
\[ R = q^{\frac{H \otimes H}{2}}\bigg(\sum_{n=0}^\infty \frac{(1-q^{-2})^n}{[n]!}(qE\otimes F)^n\bigg) \]
where 
\[ [n]=\frac{1-q^{-2n}}{1-q^{-2}} \qquad \text{and} \qquad [n]! = [n][n-1]\dotsb[1] \]
and $q^{\frac{1}{2}H\otimes H}(v \otimes v')=q^{\frac{1}{2}\lr{ v,v' }}$, 
as shown in \cite{MR1381692}*{86}.
\end{example}

\begin{proposition}\label{prop:quasi-triangular_comodule}
Let $B$ be a bialgebra with quasi-triangular structure $R$, and let $V$ be a left $B$-module.  Then $V$ is a left $B$-comodule with either of the coactions
\[ \delta_+(v) = R^{(2)}\otimes (R^{(1)} \actleft v),\quad \delta_-(v) = (R^{-1})^{(1)}\otimes \big((R^{-1})^{(2)} \actleft v\big). \]
\end{proposition}

\begin{definition}
Let $B$ be a bialgebra.  A dual quasi-triangular structure on $B$ is a convolution-invertible map $\cR\colon B\otimes B\to k$ such that
\begin{align*}
	\cR(ab\otimes c)&=\cR(a\otimes c_{(1)})\cR(b\otimes c_{(2)})\\
	\cR(a\otimes bc)&=\cR(a_{(1)}\otimes c)\cR(a_{(2)}\otimes b)\\
	b_{(1)}a_{(1)}\cR(a_{(2)}\otimes b_{(2)})&=\cR(a_{(1)}\otimes b_{(1)})a_{(2)}b_{(2)}
\end{align*}
for all $a,b,c\in B$.
\end{definition}

\begin{example}
The bialgebra $\CC_q[\rSL_n]$ has dual quasi-triangular structure given by
\begin{equation*}
	\cR(X_{ij}\otimes X_{k\ell}) = 
	\begin{cases}
		q				&\text{if $i=j = k=\ell$,}\\
		1				&\text{if $i=j \neq k=\ell$,}\\
		(q-q^{-1})		&\text{if $i=\ell < j=k$,}\\
		0				&\text{otherwise,}
	\end{cases}
\end{equation*}
as shown in \cite{MR1381692}*{133}.
\end{example}

\begin{definition}
Let $\cC$ be a category with an associative tensor product.  We say that $\cC$ is \emph{braided} if it is provided with functorial isomorphisms
\begin{equation*}
	\psi_{U,V}\colon U\otimes V\to V\otimes U
\end{equation*}
that satisfy $\psi_{U\otimes V,W}=\psi_{U,W}\circ\psi_{V,W}$ and $\psi_{U,V\otimes W}=\psi_{U,W}\circ\psi_{U,V}$ for all objects $U,V,W$.
\end{definition}

\begin{proposition}\label{prop:dual_quasi-triangular_module}
Let $B$ be a bialgebra with dual quasi-triangular structure $\cR$.  Then every left $B$-comodule is also a left $B$-module with action
\begin{equation*}
	b \actleft v = \cR\big(b \otimes v^{(-1)}\big)v^{(0)}.
\end{equation*}
Furthermore, the category of left $B$-comodules is braided with
\begin{equation*}
	\psi(u \otimes v) = \big(u^{(-1)}\actleft v\big) \otimes u^{(0)}.
\end{equation*}
\end{proposition}

\subsection{Yetter-Drinfeld modules and the quantum double}\label{section:the_quantum_double}

\begin{definition}
Let $H$ be a bialgebra.  Then $V$ is a \emph{left Yetter-Drinfeld $H$-module} if $V$ is both a left $H$-module and a left $H$-comodule and if the action and coaction satisfy the relation 
\begin{equation*}
	h_{(1)}v^{(-1)}\otimes (h_{(2)}\actleft v^{(0)}) = (h_{(1)}\actleft v)^{(-1)}h_{(2)}\otimes(h_{(1)}\actleft v)^{(0)}.
\end{equation*}
If $H$ is a Hopf algebra, then this can be written
\begin{equation*}
	\delta(h\actleft v) = h_{(1)}v^{(-1)}S(h_{(3)})\otimes (h_{(2)}\actleft v^{(0)}).
\end{equation*}
\end{definition}

\begin{proposition}
Let $H$ be a Hopf algebra.
\begin{enumerate}
\item If $H$ is quasi-triangular, then the coaction of Proposition \ref{prop:quasi-triangular_comodule} makes every $H$-module into a Yetter-Drinfeld module.
\item If $H$ is dual-quasi-triangular, then the action of Proposition \ref{prop:dual_quasi-triangular_module} makes every $H$-comodule into a Yetter-Drinfeld module.
\end{enumerate}
\end{proposition}

\begin{proposition}
Let $H$ be a bialgebra with a braided category $\cC$ as in Proposition \ref{prop:dual_quasi-triangular_module}.  If the objects of $\cC$ are Yetter-Drinfeld modules, then $(1 \otimes \psi) \circ \delta = \delta \circ \psi$.
\end{proposition}

\begin{definition}
Let $\cC$ be a monoidal category.  The \emph{Drinfeld center} of $\cC$ is the monoidal category whose objects are objects $X$ of $\cC$ together with a natural isomorphism $\psi_X\colon X\otimes Y\to Y\otimes X$, for any other object $Y$, such that $\psi_{X\otimes Y}=(\id\otimes\psi_Y)\circ(\psi_X\otimes\id)$ for all $X,Y$.
\end{definition}

\begin{proposition}
Let $H$ be a bialgebra and let $\cC$ be the category of left $H$-modules.  Then an object of $\cC$ is a Yetter-Drinfeld module if and only if it belongs to the Drinfeld center of $\cC$.
\end{proposition}

To prove this, we already saw in Proposition \ref{prop:dual_quasi-triangular_module} how the coaction can be used to produce a twisting.  On the other hand, since $H$ is a left $H$-module where the action is left multiplication, we can define $\delta(v)=\psi_V(v\otimes 1)$ where $1\in H$.

If $H$ is a Hopf algebra and $\cC$ is the category of left $H$-modules, then the Drinfeld center of $\cC$ is also the category of left modules over a Hopf algebra related to $H$, called the Drinfeld double, or quantum double, of $H$, which we now define.

\begin{definition}
Let $H$ and $C$ be Hopf algebras over a field $k$, and let $\phi\colon C\otimes H\to k$ be a Hopf pairing.  We define a Hopf algebra called the \emph{quantum double} $\rD(H)$ as follows.  As a coalgebra, $\rD(H)=C\otimes H$ with the tensor coalgebra structure of Lemma \ref{lem:tensor_coalgebra}, and thus both $C$ and $H$ are sub-coalgebras of $\rD(H)$.  As an algebra, $C^\op$ and $H$ are subalgebras.  Specifically, the multiplication $\cdot$ of $\rD(H)$ is given by $c\cdot c'=c'c$ for all $c,c'\in C$, by $h\cdot h'=hh'$ for all $h,h'\in H$, and by the cross-relation
\begin{equation*}
	h\cdot c=c_{(2)}\cdot h_{(2)}\phi(c_{(1)},Sh_{(1)})\phi(c_{(3)},h_{(3)}).
\end{equation*}
\end{definition}
We note that according to Proposition \ref{prop:pairing_action}, the above is equivalent to
\begin{equation*}
	h\cdot c=(h_{(3)}\actleft c \actright Sh_{(1)})\cdot h_{(2)}.
\end{equation*}

\begin{example}\label{ex:duqgln_cross_relations}
The Hopf algebras $\rU_q(\fgl_n)$ and $\CC_q[\rGL_n]$ are dually paired as shown in Example \ref{ex:pairing}.  Then as a coalgebra, $\rD(\rU_q(\fgl_n))=\CC_q[\rGL_n] \otimes \rU_q(\fgl_n)$ as defined in Lemma \ref{lem:tensor_coalgebra}.  Also, $\rD(\rU_q(\fgl_n))$ has subalgebras $\rU_q(\fgl_n)$, as presented in Definition \ref{def:quantum_enveloping_algebra}, and $\CC_q[\rGL_n]^\op$, as given in Definition \ref{def:quantum_matrices_gln} but with opposite multiplication.  The cross-relations are
\begin{align*}
	E_iX_{k\ell}
	&= (q\delta_{i,\ell}+q^{-1}\delta_{i+1,\ell})X_{k\ell}E_i-(q^2\delta_{i,\ell}+\delta_{i+1,\ell})\delta_{i,k}X_{k+1,\ell}K_i+\delta_{i+1,\ell}X_{k,\ell-1}\\
	F_iX_{k\ell}
	&= (q^{-1}\delta_{i+1,k}+q\delta_{i,k})X_{k\ell}F_i +(q^{-1}\delta_{i+1,k}+q\delta_{i,k})\delta_{i,\ell}X_{k,\ell+1}K_i^{-1}-q^{-1}\delta_{i+1,k}X_{k-1,\ell}\\
	K_iX_{k\ell}
	&= (q^2\delta_{i+1,k,\ell+1}+\delta_{i,k,\ell}+\delta_{i+1,k,\ell}+q^{-2}\delta_{i+1,k+1,\ell})X_{k\ell}K_i
\end{align*}
Here $\delta_{i,j,k}=1$ if $i=k=j$, and $\delta_{i,j,k}=0$ otherwise.
\end{example}

We note that the quantum determinant of $\CC_q[\rGL_n]$ is central and group-like here, and the action of $\rU_q(\fgl_n)$ on it is by the counit.

\begin{example}\label{ex:sl2_double_relations}
In particular, the algebra $\rD(\rU_q(\fsl_2))$ has cross-relations
\begin{align*}
	Ea&=qaE-q^2cK		&	Ec&=qcE\\
	Eb&=q^{-1}bE-dK+a	&	Ed&=q^{-1}dE+c\\
	Fa&=qaF+qbK^{-1}	&	Fc&=q^{-1}cF+q^{-1}dK^{-1}-q^{-1}a\\
	Fb&=qbF				&	Fd&=q^{-1}dF-q^{-1}b\\
	Ka&=aK				&	Kc&=q^2cK\\
	Kb&=q^{-2}bK		&	Kd&=dK
\end{align*}
It will be helpful in our investigation of highest-weight vectors that $c$ quasi-commutes with $E$ (See Lemma \ref{lem:c_action_preserves_hw}).
\end{example}

We will also be interested in the dual $\rD(H)^*$ of the quantum double.  If $H$ is infinite-dimensional, then this Hopf algebra is very complicated.  As we will see, even knowing the finite dual $\rD(H)^\circ$ is as complicated as knowing the entire category of finite-dimensional $\rD(H)$-modules.  However, there is a Hopf subalgebra of $\rD(H)^*$ which is equal to $\rD(H)^*$ if $H$ is finite-dimensional but is much easier to describe when $H$ is infinite-dimensional.

\begin{proposition}\label{prop:h_otimes_c}
Let $H$ and $C$ be Hopf algebras over a field $k$, and let $\phi\colon C\otimes H\to k$ be a Hopf pairing.  There is a subalgebra $H\otimes C\subset \rD(H)^*$ which has the tensor algebra structure of Lemma \ref{lem:tensor_algebra}. (See \cite{MR1381692}*{334 and 362}.)
\end{proposition}

We will suppress the tensor symbol when writing elements of $H\otimes C$.

\begin{proposition}\label{prop:H-mod_to_DH-mod}
Let $H$ be a Hopf algebra with quasi-triangular structure $R$.  There is an embedding $\Phi_R\colon \text{$H$-mod} \hookrightarrow \text{$\rD(H)$-mod}$ which gives each $H$-module the structure of a $\rD(H)$-module.
\end{proposition}

To prove the above proposition we use the coaction $\delta(v) = R^{(2)} \otimes (R^{(1)} \actleft v)$ as in Proposition \ref{prop:quasi-triangular_comodule} and then define the action of elements of $c$ by 
\begin{equation*}
	c \actleft v = \phi\big(c,v^{(-1)}\big)v^{(0)}
\end{equation*}
where $\phi$ is the pairing between $C$ and $H$.
In the same way we could show using Proposition \ref{prop:dual_quasi-triangular_module} that if $H$ has a dual quasi-triangular structure then there is an embedding of $H$-comod into $\rD(H)$-mod.

Because the Hopf algebra $\rU_q(\fsl_2)$ is quasi-triangular, every $\rU_q(\fsl_2)$-module is a $\rD(\rU_q(\fsl_2))$-module.  Because the quasi-triangular structure $R$ is invertible, we can use either $R$ or $R^{-1}$ to construct a $\rD(\rU_q(\fsl_2))$-module from an $H$-module.

\begin{proposition}\label{prop:n_simple_1-dimensional_modules}
There are exactly $n$ one-dimensional $\rD(\rU_q(\fsl_n))$-modules.
\end{proposition}

\begin{proof}
A one-dimensional left $\rD(\rU_q(\fsl_n))$-module is equivalent to an algebra homomorphism $\phi\colon \rD(\rU_q(\fsl_n))\to k$.  We know that the only one-dimensional $\rU_q(\fsl_n)$-module is the trivial module, where
\begin{equation*}
	\phi(K_i)=\phi(K_i^{-1})=1\quad\text{and}\quad \phi(E_i)=\phi(F_i)=0.
\end{equation*}
We refer now to the cross-relations in Example \ref{ex:duqgln_cross_relations}.  By commuting $E_i$ past $X_{k,i+1}$, we find that $\phi(X_{k,i})=0$ if $i\neq k$ and $\phi(X_{i+1,i+1})=\phi(X_{i,i})$.  By commuting $F_i$ past $X_{i+1,n}$ we find that $\phi(X_{i,n})=0$.  Thus
\begin{equation*}
	\phi(X_{k,k})=\phi(X_{\ell,\ell})\quad\text{and}\quad \phi(X_{k,\ell})=0\quad\text{for all $k\neq \ell$.}
\end{equation*}
However, the quantum determinant implies that $\prod_{k=1}^n \phi(X_{k,k})=1$.  Thus we can choose $\phi(X_{1,1})$ to be any $n$th root of unity.
\end{proof}

\section{Some semisimplicity results}\label{C:Preliminary_Results}

\subsection{A correspondence of subcategories and sub-bimodules}\label{section:correspondence}

If $B$ is an algebra and $V$ is a left $B$-module, then $V^*$ is a right $B$-module, and we can define a map $\beta_V\colon V\otimes V^*\to B^*$ so that $\beta_V(v\otimes f)$ is a linear function on $B$ given by
\begin{equation}\label{eqn:matrix_coefficients_2}
	\beta_V(v\otimes f)(b)=f(b\actleft v)=(f\actright b)(v)
\end{equation}
for all $v\in V$, $f\in V^*$.  We refer to $\beta_V(v\otimes f)$ as a \emph{matrix coefficient}.

\begin{lemma}\label{lemma:isom_classes_collapse}
If $U$ and $V$ are isomorphic $B$-modules, then $\beta_U(U\otimes U^*)=\beta_V(V\otimes V^*)$.
\end{lemma}

\begin{proof}
Let $\phi\colon V\to U$ be an isomorphism of $B$-modules.  Then
\begin{equation*}
	\begin{split}
		\beta_U(\phi v\otimes \phi^*f)(b)
		&= \phi^*f(b\actleft \phi v)\\
		&= \phi^*f\big(\phi (b\actleft v)\big)\\
		&= f(b\actleft v)\\
		&= \beta_V(v\otimes f)
	\end{split}
\end{equation*}
for any $v\in V$ and $f\in V^*$.
\end{proof}

\begin{lemma}\label{lemma:the_map_beta}
The maps $\{\beta_V\mid V\in \text{$B$-mod}\}$ in \eqref{eqn:matrix_coefficients_2} are morphisms of $B$-bimodules, and $\beta_{U\oplus V}=\beta_U\oplus\beta_V$.
\end{lemma}

\begin{proof}
Let $b,b'\in B$, $v\in V$, and $f\in V^*$.  We see that
\begin{equation*}
	\begin{split}
		\beta_V\big((b'\actleft v)\otimes f\big)(b)
		&= f\big(b\actleft(b'\actleft v)\big)\\
		&= f(bb'\actleft v)\\
		&= \big(f\actright(bb')\big)(v)\\
		&= \beta_V(v\otimes f)(bb')\\
		&= \big(b'\actleft \beta_V(v\otimes f)\big)(b)
	\end{split}
\end{equation*}
and
\begin{equation*}
	\begin{split}
		\beta_V\big(v\otimes(f\actright b')\big)(b)
		&= (f\actright b')(b\actleft v)\\
		&= f(b'b\actleft v)\\
		&= (f\actright b'b)(v)\\
		&= \beta_V(v\otimes f)(b'b)\\
		&= \big(\beta_V(v\otimes f)\actright b'\big)(b).
	\end{split}
\end{equation*}
Let $U$ and $V$ be $B$-modules, and let $u\in U$, $v\in V$, $f\in U^*$, and $g\in V^*$.  Then
\begin{equation*}
	\begin{split}
		\beta_{U\oplus V}\big((u\oplus v)\otimes(f\oplus g)\big)(b)
		&= (f\oplus g)\big(b\actleft(u\oplus v)\big)\\
		&= f(b\actleft u)+g(b\actleft v)\\
		&= \beta_U(u\otimes f)(b)+\beta_V(v\otimes g)(b)
	\end{split}
\end{equation*}
so $\beta_{U\oplus V}=\beta_U\oplus\beta_V$.
\end{proof}

This means that $\beta$'s effectively ignore multiplicity:

\begin{corollary}\label{cor:beta_kills_multiplicity}
The image of $\beta_{V\oplus V}$ is equal to the image of $\beta_V$.
\end{corollary}

\begin{lemma}\label{lemma:beta_injective_on_simples}
If $V$ is a simple $B$-module, then $\beta_V$ is injective.
\end{lemma}

\begin{proof}
We have $\beta_V(v\otimes f)(B)=0$ if and only if $f(B\actleft v)=0$.  This is true if and only if $v=0$ or $f=0$, i.e.\ if and only if $v\otimes f = 0$.
\end{proof}

We now define a correspondence between additive subcategories of $B$-mod and sub-bimodules of $B^*$.
For any additive category $\cC$ of $B$-modules, we denote by $B_\cC^*$ the span of the images of $\{\beta_V\mid V\in\cC\}$.  On the other hand, given a sub-bimodule $D$ of $B^*$, we define $\cC(D)$ to be the full subcategory of objects $V\in\cC$ such that $\beta_V(V\otimes V^*)\subset D$.

\begin{remark}
Lemma \ref{lemma:the_map_beta} shows why when defining our correspondence we assume that $D$ is a bimodule and $\cC$ is additive.  However, it is still not clear whether the containments $\cC\subseteq\cC(B_\cC^*)$ and $B_{\cC(D)}^*\subseteq D$ are equalities.
\end{remark}

\begin{proposition}
If $B$ is a bialgebra and $\cC$ is monoidal, then $B_\cC^*$ is a subalgebra of $B^*$.  On the other hand, $D$ is a subalgebra of $B^*$ if and only if $\cC(D)$ is monoidal.
\end{proposition}

\begin{proof}
Let $U$ and $V$ be objects of $\cC$, and let $u\in U$, $v\in V$, $f\in U^*$, and $g\in V^*$.  Then
\begin{equation*}
	\begin{split}
		\beta_{U\otimes V}\big((u\otimes v)\otimes(g\otimes f)\big)(b)
		&= (f\otimes g)\big(b\actleft(u\otimes v)\big)\\
		&= f(b_{(1)}\actleft u)\cdot g(b_{(2)}\actleft v)\\
		&= \beta_U(u\otimes f)(b_{(1)})\cdot\beta_V(v\otimes g)(b_{(2)})
	\end{split}
\end{equation*}
so $\beta_{U\otimes V}=\beta_U\beta_V$.
\end{proof}

\begin{definition}
An element of $B^*$ is \emph{locally finite} if it generates a finite-dimensional bimodule.  If $D$ is a sub-bimodule of $B^*$, we denote by $D_f$ the sub-bimodule of locally finite elements of $D$.
\end{definition}

Now the product of two locally finite elements belongs to the tensor product of their respective finite-dimensional submodules, which proves the following lemma.

\begin{lemma}\label{lem:lf_subalgebra}
If a sub-bimodule $D$ of $B^*$ is a subalgebra, then the sub-bimodule $D_f$ of locally finite elements is in fact a subalgebra.
\end{lemma}

\subsection{A Peter-Weyl-type theorem}\label{section:PW_theorem}

In this section we present a Peter-Weyl-type theorem relating semisimplicity of $\cC$ with a Peter-Weyl decomposition of $B_\cC^*$.  The author failed to find a complete reference for this theorem in the literature, although one direction of the implication is well known and for this part the author appreciated the proof given in a lecture by David Jordan \cite{jordan2011}.

\begin{definition}
Let $D\subset B^*$ be a sub-bimodule.  We say that $D$ has a Peter-Weyl decomposition if
\begin{equation*}
	D = \bigoplus \beta_V(V\otimes V^*)
\end{equation*}
as an internal direct sum over all isomorphism classes of simple objects $V\in \cC(D)$.  Lemma \ref{lemma:isom_classes_collapse} shows this is well defined.
\end{definition}

\begin{theorem}\label{thm:PW_thm}
Let $B$ be an algebra and let $\cC$ be an Abelian category of finite-dimensional $B$-modules.  Then $\cC$ is semisimple if and only if $B_\cC^*$ has a Peter-Weyl decomposition.
\end{theorem}

Before proving the theorem, we note the following well known result which is proved by induction on the length of a Krull-Schmidt decomposition.

\begin{lemma}\label{lem:KrullSchmidt_and_Ext}
Let $B$ be an algebra and let $\cC$ be an Abelian category of finite-dimensional $B$-modules.  Then $\cC$ is semisimple if and only if $\Ext^1(U,V)=0$ for all simple $B$-modules $U$ and $V$.
\end{lemma}

\begin{proof}[Proof of Theorem \ref{thm:PW_thm}]
Suppose that $B_\cC^*$ has a Peter-Weyl decomposition, and suppose by way of contradiction that $\cC$ has an indecomposable object $V$ and a short exact sequence
\[ 0 \to V_1 \to V \to V_2 \to 0 \]
where $V_1$ and $V_2$ are both simple.  Now $V$ must be cyclic; if not, then any $v\in V\setminus V_1$ would generate a complement to $V_1$, and we have assumed that $V$ is indecomposable.

Now the dual short exact sequence
$ 0 \to V_2^* \to V^* \to V_1^* \to 0 $
has the same properties.  Choose a cyclic vector $f\in V^*$.  We define $\iota_f\colon V\to V\otimes V^*$ by $ \iota_f(v) = v\otimes f $.  We claim that $ \beta_V\circ \iota_f $ is an injective morphism of $B$-modules.  Indeed, $\beta_V(v\otimes f)$ is the zero map if and only if $(f\actright b)(v)=0$ for all $b\in B$, and since $f$ generates $V^*$ this implies $v=0$.  Thus $\beta_V\circ\iota_f$ embeds $V$ into $B_\cC^*$, which is semisimple, so $V$ is semisimple, contradicting our assumption that $V$ was indecomposable.  Therefore $\cC$ is semisimple by Lemma \ref{lem:KrullSchmidt_and_Ext}.

Suppose now that $\cC$ is semisimple.  Let $V\cong\bigoplus_{i=1}^n\big(\bigoplus_{j=1}^{n_i} V_i\:\big)$ be a decomposition of a module $V\in\cC$ as a sum of simple modules $V_i$.  By Lemma \ref{lemma:the_map_beta} and Corollary \ref{cor:beta_kills_multiplicity}, the image of $\beta_V$ is equal to $\sum \beta_{V_i}(V_i \otimes V_i^*)$.  By Lemma \ref{lemma:beta_injective_on_simples}, the sum is direct.
\end{proof}

\subsection{Semisimplicity via a Casimir element}\label{section:casimir}

In this section we present a theorem proving the semisimplicity of certain representations of Hopf algebras when a Casimir element is available.  The proof is a straightforward generalization of proofs given elsewhere.  For example, see \cite{MR0323842}*{28} for semisimple Lie algebras and \cite{MR953821}*{587--589} for $\rU_q(\fsl_2)$.

Let $H$ be a Hopf algebra.  Recall that if $V$ is an irreducible left $H$-module and $c\in H$ belongs to the center of $H$, then $c$ acts on $V$ as multiplication by some scalar.

\begin{theorem}\label{thm:peter-weyl}
Let $H$ be a Hopf algebra, let $\cC$ be an Abelian category of finite-dimensional left $H$-modules which is closed under extension, and let $\one$ denote the trivial one-dimensional $H$-module.  Suppose that there exists an element $c$ from the center of $H$ with the following property: For any simple $V$ in $\cC$, $c\actleft V=0$ if and only if $V\cong\one$.  Suppose furthermore that $\Ext^1(\one,\one)=0$.  Then every $H$-module in $\cC$ is semisimple.
\end{theorem}

We call the element $c$ a \emph{Casimir} element of $H$.  Given a finite-dimensional left $H$-module $V$, the strategy for the proof is to show that for any submodule $W\subset V$ there is another submodule $W'\subset V$ such that $V=W\oplus W'$.  We first consider a couple of special cases.

\begin{lemma}\label{lem:peter-weyl-1}
Let $H$ be a Hopf algebra as described in Theorem \ref{thm:peter-weyl}, and let $V\in\cC$ be a finite-dimensional left $H$-module.  If $W\subset V$ is an irreducible submodule with $V/W\cong\one$, then there exists another submodule $W'\subset V$ such that $V=W\oplus W'$.
\end{lemma}

\begin{proof}
The Casimir element $c$ satisfies $c\rhd\bar{v}=0$ for all $\bar{v}\in V/W$.  If $W\cong\one$, then $V\cong\one\oplus\one$ since $\Ext^1(\one,\one)=0$.  If $W\ncong\one$, then we know that $c\rhd W\neq 0$ by the hypothesis of Theorem \ref{thm:peter-weyl}.  Therefore the submodule $\ker(c)$ of $V$ satisfies $\ker(c)\cap W=0$, so $V=W\oplus\ker(c)$.
\end{proof}

\begin{lemma}\label{lem:peter-weyl-2}
Let $H$ be a Hopf algebra as described in Theorem \ref{thm:peter-weyl}, and let $V\in\cC$ be a finite-dimensional left $H$-module.  If $W\subset V$ is a submodule with $V/W\cong\one$, then there exists another submodule $W'\subset V$ such that $V=W\oplus W'$.
\end{lemma}

\begin{proof}
If $W$ is irreducible, then this follows from Lemma \ref{lem:peter-weyl-1}.  So, suppose that $W$ has a proper nonzero submodule $U\subset W$.  Then we may write the short exact sequence of $H$-modules
\begin{equation*}
	0 \to W/U \to V/U \to V/W \to 0.
\end{equation*}
Now $\dim(W/U)<\dim(W)$.  We use induction on the dimension of $W$, noting that in the base case $W$ is irreducible.  So by hypothesis the short exact sequence splits and there is a submodule $U'\subset V$ such that $V/U=W/U\oplus U'/U$.  Note that $U'/U\cong\one$ since $V/W\cong\one$.  We now write the short exact sequence of $H$-modules
\begin{equation*}
	0 \to U \to U' \to U'/U \to 0.
\end{equation*}
Now $\dim(U)<\dim(W)$, so by hypothesis the short exact sequence splits and there is a submodule $W'\subset U'$ such that $U'=U\oplus W'$.  It follows that $V/U=W/U\oplus W'$.  Thus $W\cap W'=0$, and we conclude that $V=W\oplus W'$.
\end{proof}

We are now ready to prove Theorem \ref{thm:peter-weyl}.

\begin{proof}[Proof of Theorem \ref{thm:peter-weyl}]
Let $V\in\cC$ be a finite-dimensional left $H$-module, and suppose $W\subset V$ is a proper non-zero submodule. We know that $\Hom_k(V,W)$ is a left $H$-module with action given by
\begin{equation*}
	(h\actleft \phi)(v)=h_{(1)}\actleft \phi(Sh_{(2)}\actleft v).
\end{equation*}
We define two subspaces $L$ and $L'$ of $\Hom_k(V,W)$ as follows:
\begin{align*}
	L	&=\{\phi\mid \text{$\exists\:f(\phi)\in k$ such that $\phi(w)=f(\phi) w$ for all $w\in W$}\},\\
	L'	&=\{\phi\mid \text{$\phi(w)=0$ for all $w\in W$}\}.
\end{align*}
We wish to show that $L$ and $L'$ are submodules of $\Hom_k(V,W)$.  Let $h\in H$, $\phi\in L$, and $w\in W$.  Then
\begin{equation*}
	\begin{split}
		(h\actleft \phi)(w)
		&=h_{(1)}\actleft \phi(Sh_{(2)}\actleft w)\\
		&=h_{(1)}\actleft f(\phi)(Sh_{(2)}\actleft w)\\
		&=f(\phi)(h_{(1)}Sh_{(2)}\rhd w)\\
		&=f(\phi)\varepsilon(h)w.
	\end{split}
\end{equation*}
Thus $(h\rhd \phi)\in L$, so $L$ is an $H$-module.  Similarly $L'$ is an $H$-module.  We note that $L/L'\cong\one$, so by Lemma \ref{lem:peter-weyl-2} there is an $H$-module $L''$ such that $L=L'\oplus L''$.  Let us choose some nonzero $\phi\in L''$, scaled as necessary so that $f(\phi)=1$.  Since $L''$ is an $H$-module we have $\big[\big(h-\counit (h)\big)\actleft\phi\big]\in L''$ for all $h\in H$.  But our calculation above shows that $\big[\big(h-\counit(h)\big)\actleft\phi\big](w)=0$ for all $w\in W$, so $\big[\big(h-\counit (h)\big)\actleft\phi\big]\in L'$.  Since $L'$ and $L''$ have trivial intersection, $\big[\big(h-\counit (h)\big)\actleft\phi\big]=0$.  That is,
\[ (h\actleft \phi)(v)=\counit(h)\phi(v) \]
for all $v\in V$.  Thus $\phi$ is not merely $k$-linear, but is a homomorphism of $H$-modules.  It is surjective since it belongs to $L$.  Therefore $V=W\oplus\ker(\phi)$.
\end{proof}

\subsection{Some semisimple categories of $\rD(\rU_q(\fg))$-modules}\label{section:Vplusminus}

In this section we demonstrate a method of constructing simple $\rD(H)$-modules, where $H$ is a bialgebra, which the author learned from Victor Ostrik.  We stated in Proposition \ref{prop:quasi-triangular_comodule} that if $H$ has quasi-triangular structure $R$, then we can construct coactions
\begin{equation*}
	\delta_+(v) = R^{(2)} \otimes (R^{(1)} \actleft v)\quad\text{and}\quad
	\delta_-(v) = (R^{-1})^{(1)} \otimes \big( (R^{-1})^{(2)} \actleft v \big)
\end{equation*}
for any $H$-module.  We refer to the resulting $\rD(H)$-modules (See Proposition \ref{prop:H-mod_to_DH-mod}) as $V^+$ and $V^-$, respectively.

\begin{lemma}\label{lem:v_plus_neq_v_minus}
Let $H$ be a bialgebra and let $U$ and $V$ be non-isomorphic simple $H$-modules.  Then $U^-\ncong V^+$, and furthermore $V^-\cong V^+$ if and only if $\delta_-=\delta_+$.
\end{lemma}

\begin{proof}
That $U^-\ncong V^+$ is obvious since $U^-$ and $V^+$ retain the $H$-module structures of $U$ and $V$, respectively, and $H$ is a subalgebra of $D(H)$.

Now assume that $f\colon V\to V$ is an isomorphism such that
\begin{equation*}
	\delta_-=(\id_H\otimes f)^{-1}\circ \delta_+ \circ f.
\end{equation*}
Since $V$ is simple, Schur's Lemma implies that $f(v)=cv$ for all $v\in V$, where $c$ is a non-zero constant.  Therefore, $\delta_-=\delta_+$.
\end{proof}

\begin{lemma}\label{lem:two_simples_4}
Let $H$ be a quasi-triangular Hopf algebra.  Suppose that $V\otimes V^*$ is semisimple for any simple $H$-module $V$, that the category of finite-dimensional left $H$-modules is semisimple, and that $\delta_-\neq\delta_+$ except on the trivial $H$-module.  Let $U$ and $V$ be simple $H$-modules.  Then the $\rD(H)$-module $U^+\otimes V^-$ is simple.  Furthermore, the tensor category generated by all such $U^+\otimes V^-$ is semisimple.
\end{lemma}

\begin{proof}
Let $U$ and $V$ be simple $H$-modules.  Lemma \ref{lem:v_plus_neq_v_minus} implies that $\Hom_{\rD(H)}(V^+,V^-)$ is nonzero if and only if $V$ is the trivial module.  We have
\begin{equation*}
	\begin{split}
		\End_{\rD(H)}(U^+\otimes V^-)
		&=\Hom_{\rD(H)}(U^+\otimes V^-,U^+\otimes V^-)\\
		&=\Hom_{\rD(H)}\big(U^+\otimes (U^+)^* , V^+\otimes (V^-)^*\big)\\
		&=\Hom_{\rD(H)}\big((U \otimes U^*)^+ , (V\otimes V^*)^-\big)
	\end{split}
\end{equation*}
which is $\CC(q)$ since the only contribution is from the trivial submodules of $(U\otimes U^*)^+$ and $(V\otimes V^*)^-$.  Thus $U^+\otimes V^-$ is simple.

Now let $U$, $V$, $W$, and $Y$ be left $H$-modules.  We have
\begin{equation*}
	(U^+\otimes V^-)\otimes(W^+\otimes Y^-)
	\cong U^+\otimes W^+\otimes V^-\otimes Y^-
	\cong (U\otimes W)^+\otimes(V\otimes Y)^-
\end{equation*}
so the lemma is proved.
\end{proof}

\begin{definition}\label{defn:V_lambda_mu}
Recall that $\rU_q(\fg)$ is quasi-triangular with simple modules $V_\lambda$.  We define $V_{\lambda,\mu}=V_\lambda^+\otimes V_\mu^-$.
\end{definition}

\begin{corollary}
The $\rD(\rU_q(\fsl_2))$-modules $V_{\lambda,\mu}$ are simple, and the category they generate is semisimple.
\end{corollary}

\begin{proof}
Let $V$ be a simple, non-trivial $\rU_q(\fsl_2)$-module.  Let $v\in V$ be a highest-weight vector.  Since $V$ is non-trivial, we have $F\actleft v\neq 0$.  Given the quasi-triangular structure from Example \ref{ex:r_for_uqsl2}, we find that $v$ generates a 1-dimensional $H$-comodule under $\delta_+$, but not under $\delta_-$.  Thus $\delta_-\neq\delta_+$, so the hypotheses of Lemma \ref{lem:two_simples_4} are satisfied.
\end{proof}

The following conjectures are due to the author's conversations with Victor Ostrik.
\begin{conjecture}
The category of finite-dimensional $\rD(\rU_q(\fg))$-modules is semi\-simple.
\end{conjecture}
\begin{conjecture}
The simple $\rD(\rU_q(\fg))$-modules are, up to isomorphism, of the form $V_{\lambda,\mu}\otimes U_0$ where $\lambda$ and $\mu$ are dominant integral $\fg$-weights and $U_0$ belongs to the (finite) set of one-dimensional $\rD(\rU_q(\fg))$-modules.
\end{conjecture}

These conjectures are extremely difficult even for $\fg=\fsl_2$.  However, we find that they are consistent with our main results in Section \ref{C:Calculations}.

\section{Main results}\label{C:Calculations}

Proofs for results presented in this section may be found in Section \ref{C:Proofs}.

\subsection{The actions of $\rD(H)$ on $H\otimes C$}\label{section:double_actions}

It follows from Lemma \ref{lem:dual_bialgebra} that $\rD(H)^*$ is a $\rD(H)$-bimodule, but the actions are very complicated when $H$ is infinite-dimensional.  However, there is a sub-bimodule that we are able to describe.
\begin{theorem}\label{thm:dual_is_a_bimodule}
The subalgebra $H\otimes C\subset \rD(H)^*$ of Proposition \ref{prop:h_otimes_c} is closed under both the left and right actions of $\rD(H)$ described in Section \ref{section:hopf_pairing}, and thus $H\otimes C$ is a bimodule algebra over $\rD(H)$.  The left and right actions on generators are given by
\begin{align*}
	\makebox[1.3ex][c]{$c$} \blactleft \makebox[1.3ex][l]{$\bar{h}$}
	&=(\bar{h}\actright c_{(2)})\cdot Sc_{(1)}c_{(3)}
		&\makebox[1.3ex][l]{$\bar{h}$} \blactright \makebox[1.3ex][l]{$c$} 
		&=c\actleft \bar{h}\\
	\makebox[1.3ex][c]{$c$} \blactleft \makebox[1.3ex][l]{$\bar{c}$}
	&=Sc_{(1)}\bar{c}c_{(2)}
		&\makebox[1.3ex][l]{$\bar{c}$} \blactright \makebox[1.3ex][l]{$c$} 
		&=\varepsilon(c)\bar{c}\\
	\makebox[1.3ex][c]{$h$} \blactleft \makebox[1.3ex][l]{$\bar{h}$}
	&=\varepsilon(h)\bar{h}
		&\makebox[1.3ex][l]{$\bar{h}$} \blactright \makebox[1.3ex][l]{$h$} 
		&=Sh_{(1)} \bar{h} h_{(2)}\\
	\makebox[1.3ex][c]{$h$} \blactleft \makebox[1.3ex][l]{$\bar{c}$}
	&=h\actleft \bar{c}
		&\makebox[1.3ex][l]{$\bar{c}$} \blactright \makebox[1.3ex][l]{$h$} 
		&=Sh_{(1)} h_{(3)}\cdot (\bar{c}\actright h_{(2)})
\end{align*}
\end{theorem}

We use the solid triangles $\blactleft$ and $\blactright$ to distinguish these actions from the actions of $C$ and $H$ on each other, for which we use $\actleft$ and $\actright$.  
We remark that $\bar{h} \blactright c =c\actleft \bar{h}$ makes sense because the subalgebra $C^{\op}\subset D(H)$ has multiplication opposite to that of $C$.

Theorem \ref{thm:dual_is_a_bimodule} gives explicit formulas for the actions of $\rD(H)$ on $H\otimes C$, so we will seek to describe its locally finite part $(H\otimes C)_f$ and thus to describe all objects of the category $\cC\big((H\otimes C)_f\big)$.  Ideally we would be able to describe the finite dual $\rD(H)^\circ$ and thus all finite-dimensional $\rD(H)$-modules, but that is much more difficult.

For examples of these actions, see Sections \ref{section:duqsln_actions} and \ref{section:duqsl2_actions}.

\subsection{Semisimplicity of certain $\rD(\rU_q(\fsl_2))$-modules}\label{section:main_results}

If $H=\rU_q(\fg)$, we recall $V_{\lambda,\mu}=V_\lambda^+\otimes V_\mu^-$ from Definition \ref{defn:V_lambda_mu}.  We recall that $\beta_{V_{\lambda,\mu}}(V_{\lambda,\mu}\otimes V_{\lambda,\mu}^*)$ is a sub-bimodule of $\rD(H)^*_f$ as shown in Section \ref{section:PW_theorem}.

In the remainder of this section, we let $H=\rU_q(\fsl_2)$, $C=\CC_q[\rSL_2]$, and $\cH=(\dualdouble)_f$.  In this case, $\lambda$ and $\mu$ are nonnegative integers, and there are simple sub-bimodules $\cH_{\lambda,\mu}$ corresponding to $\beta_{V_{\lambda,\mu}}(V_{\lambda,\mu}\otimes V_{\lambda,\mu}^*)$.  We claim that $\cH$ has a Peter-Weyl decomposition, namely the following.
\begin{maintheorem}\label{thm:main_theorem}
As a $\rD(H)$-bimodule,
\begin{equation}\label{eq:pw_decomp}
	\cH = \bigoplus_{\genfrac{}{}{0pt}{}{\lambda,\mu\geq 0}{\lambda-\mu \in 2\ZZ}} \cH_{\lambda,\mu}
\end{equation}
and this is a Peter-Weyl decomposition of $\cH$.
\end{maintheorem}

We recall the notation $\cC(D)$ from Section \ref{section:correspondence}.
Then Theorems \ref{thm:PW_thm} and \ref{thm:main_theorem} have the following corollary:

\begin{corollary}
The category $\cC(\cH)$ is semisimple.
\end{corollary}

We now highlight some of the results leading to Main Theorem \ref{thm:main_theorem}.

\begin{definition}
If $v\in \cH$, we say that $v$ is homogeneous of weight $(\omega_1,\omega_2)$ if there exist scalars $\omega_1$ and $\omega_2$ such that $K\blactleft v=q^{\omega_1}v$ and $v\blactright K^{-1}=q^{\omega_2}v$.
\end{definition}

In Section \ref{section:DH_bimodule_examples} we will give a presentation of the 16-dimensional simple $D(H)$-bimodule $\cH_{1,1}$ and prove the following proposition about the subspace of highest-weight bivectors ${^+}\cH_{1,1}{^+}=\{v\in \cH_{1,1}\mid E\blactleft v=v\blactright E=0\}$ (Recall Lemma \ref{lem:hw_subalgebra}).

\begin{proposition}
The space ${^+}\cH_{1,1}{^+}$ is the linear span of four vectors $\bfv_1$, $\bfv_2$, $\bfv_3$, and $\bfv_4$ of weights $(2,2)$, $(2,0)$, $(0,2)$, and $(0,0)$, respectively.
\end{proposition}

The following result provides an upper bound for $\cH$ in $H\otimes C$.

\begin{theorem}\label{thm:v1v2v3v4generate_hw}
The algebra $\HCpp$ is generated by $\{\bfv_1,\bfv_2,\bfv_3,\bfv_4\}$. It has basis
\[ \{\bfv_3^\ell \bfv_1^m \bfv_5^p \bfv_2^s \}
	\cup\{\bfv_3^\ell \bfv_1^m \bfv_6^r \bfv_2^s \}
	\cup\{\bfv_3^\ell \bfv_4^n \bfv_5^p \bfv_2^s \}
	\cup\{\bfv_3^\ell \bfv_4^n \bfv_6^r \bfv_2^s \}
\]
where $\bfv_5$ and $\bfv_6$ are the highest-weight bivectors of $\cH_{2,0}$ and $\cH_{0,2}$, respectively.
\end{theorem}

Now Lemma \ref{lem:hw_subalgebra}, Lemma \ref{lem:lf_subalgebra}, and Theorem \ref{thm:v1v2v3v4generate_hw} have the following corollary.

\begin{corollary}\label{cor:HCpp_is_locally_finite}
The algebra $\HCpp$ is locally finite.
\end{corollary}

\begin{corollary}\label{cor:Hpp_equals_HCpp}
$\Hpp=\HCpp$.
\end{corollary}

Theorem \ref{thm:v1v2v3v4generate_hw}, Corollary \ref{cor:Hpp_equals_HCpp}, and Lemma \ref{lem:hw_generate_locally_finite_modules} imply the following.

\begin{corollary}\label{cor:bimod_gen_by_hw}
As a $\rD(H)$-bimodule, $\cH$ is generated by the algebra $\lr{ \bfv_1,\bfv_2,\bfv_3,\bfv_4 }$.
\end{corollary}

The above results are all we need to prove the main theorem.  However, the algebra $\Hpp = \lr{ \bfv_1,\bfv_2,\bfv_3,\bfv_4 }$ is very interesting in its own right.  The vectors $\bfv_1$ and $\bfv_4$ are central in this algebra, but $\bfv_2$ and $\bfv_3$ have homogeneous relations in degree 4.  Namely, they have the following Serre and Verma relations:
\begin{align*}
	\bfv_2^{\,3}\bfv_3
	 - (q^2+1+q^{-2}) \bfv_2^{\,2}\bfv_3\bfv_2
	 + (q^2+1+q^{-2}) \bfv_2\bfv_3\bfv_2^{\,2}
	 - \bfv_3\bfv_2^{\,3}=0\\
	\bfv_2\bfv_3^{\,3}
	 - (q^2+1+q^{-2}) \bfv_3\bfv_2\bfv_3^{\,2}
	 + (q^2+1+q^{-2}) \bfv_3^{\,2}\bfv_2\bfv_3
	 - \bfv_3^{\,3}\bfv_2=0\\
	\bfv_3\bfv_2^{\,2}\bfv_3 - \bfv_2\bfv_3^{\,2}\bfv_2=0
\end{align*}
In other words, $\Hpp$ is isomorphic to the quotient of the positive part of $\rU_q(\fg)$ by the Verma relation $\lr{\bfv_3\bfv_2^{\,2}\bfv_3-\bfv_2\bfv_3^{\,2}\bfv_2}$, where $\fg$ is the Kac-Moody algebra with Cartan matrix
\[ \begin{pmatrix}2 & 0 & 0 & 0 \\0 & 2 & -2 & 0 \\0 & -2 & 2 & 0 \\0 & 0 & 0 & 2\end{pmatrix} \]
Theorem \ref{thm:v1v2v3v4generate_hw} gives the following result about the algebra $\Hpp$ (See Section \ref{C:Proofs}).
\begin{corollary}\label{cor:hilbert_series}
The algebra $\Hpp$ has polynomial growth with Hilbert series
\begin{align*}
	h(t)=\frac{2}{(1-t)^4}-\frac{2}{(1-t)^3}+\frac{1}{(1-t)^2}.
\end{align*}
\end{corollary}
The algebra $\Hpp$ appeared in another context in \cite{MR2852300}, and it would be very interesting to continue this direction of research.

\begin{problem}
Find a presentation for the algebra $\cH$.
\end{problem}

\begin{problem}
Find a presentation for the algebras ${^+}(\rU_q(\fsl_n)\otimes\CC_q[\rSL_n]){^+}$ and $(\rU_q(\fsl_n)\otimes\CC_q[\rSL_n])_f$ for $n>2$.
\end{problem}

We conclude this section with some conjectures for other semisimple Lie algebras.  Let $\fg$ be a semisimple Lie algebra and $G$ be the corresponding simply-connected algebraic group.

\begin{conjecture}
The highest-weight bivectors in the $\rD(\rU_q(\fg))$-bimodule $\rU_q(\fg)\otimes \CC_q[G]$ are locally-finite.
\end{conjecture}

\begin{conjecture}
As a $\rD(\rU_q(\fg))$-bimodule, there is a Peter-Weyl decomposition
\begin{equation}
	\big(\rU_q(\fg)\otimes\CC_q[G]\big)_f \cong \bigoplus \cH_{\lambda,\mu}
\end{equation}
where the sum is over all dominant weights $\lambda$ and $\mu$ such that $\lambda-\mu$ belongs to the root lattice of $\fg$.
\end{conjecture}

\begin{conjecture}
The sum $\bigoplus \cH_{\omega_i,\omega_i}$ over all fundamental weights $\omega_i$ generates $\big(\rU_q(\fg)\otimes \CC_q[G]\big)_f$ as an algebra.
\end{conjecture}

\section{Examples}\label{C:Examples}

\subsection{The actions of $\rD(\rU_q(\fsl_n))$ on $\rU_q(\fsl_n)\otimes\CC_q[\rSL_n]$}\label{section:duqsln_actions}

Using Example \ref{ex:pairing} and Theorem \ref{thm:dual_is_a_bimodule} we calculate the left and right actions of $\rD(\rU_q(\fsl_n))$ on the generators of $\rU_q(\fsl_n)\otimes\CC_q[\rSL_n]$.

From $h\blactleft\bar{h}=\counit(h)\bar{h}$, we have
\begin{equation*}
	E_i\blactleft \bar{h} = 0,\quad 
	F_i\blactleft \bar{h} = 0,\quad\text{and}\quad 
	K_i\blactleft \bar{h} = 1.
\end{equation*}

From $h\blactleft\bar{c}=h\actleft\bar{c}$, we have
\begin{gather*}
	E_i\blactleft X_{k\ell} = \delta_{i+1,\ell}X_{k,\ell-1},\quad
	F_i\blactleft X_{k\ell} = \delta_{i,\ell}X_{k,\ell+1},\\
	\text{and}\quad K_i\blactleft X_{k\ell} = (q\delta_{i,\ell}+q^{-1}\delta_{i+1,\ell})X_{k\ell}.
\end{gather*}

Let us denote $SX_{k\ell}=Y_{k\ell}$. From $c\blactleft\bar{h}=(\bar{h}\actright c_{(2)})\cdot Sc_{(1)}c_{(3)}$, we have
\begin{align*}
	X_{k\ell}\blactleft E_i &= K_iY_{ki}X_{i+1,\ell} + {\textstyle\sum_j} E_i Y_{kj}X_{j\ell},\\
	X_{k\ell}\blactleft F_i &= Y_{k,i+1}X_{i\ell} + q^{-1}F_iY_{ki}X_{i\ell} + qF_iY_{k,i+1}X_{i+1,\ell},\\
	\text{and}\quad X_{k\ell}\blactleft K_i &= qK_iY_{ki}X_{i\ell} + q^{-1}K_iY_{k,i+1}X_{i+1,\ell}.
\end{align*}

From $c\blactleft\bar{c}=Sc_{(1)}\bar{c}c_{(2)}$, we have
\begin{equation*}
	X_{k\ell}\blactleft X_{mn} = {\textstyle\sum_j}Y_{kj}X_{mn}X_{j\ell}.
\end{equation*}

From $\bar{h}\blactright h=Sh_{(1)} \bar{h} h_{(2)}$, we have
\begin{gather*}
	\bar{h}\blactright E_i = -E_iK_i^{-1}\bar{h}K_i+\bar{h}E_i,\quad
	\bar{h}\blactright F_i = -K_iF_i\bar{h}+K_i\bar{h}F_i,\\
	\text{and}\quad \bar{h}\blactright K_i = K_i^{-1}\bar{h}K_i.
\end{gather*}

From $\bar{c}\blactright h=Sh_{(1)} h_{(3)}\cdot (\bar{c}\actright h_{(2)})$, we have
\begin{align*}
	X_{k\ell}\blactright E_i &= -(q^{-1}\delta_{i+1,k}+q\delta_{i,k})E_iX_{k\ell} + \delta_{i,k}K_iX_{k+1,\ell} + E_iX_{k\ell},\\
	X_{k\ell}\blactright F_i &= -K_iF_iX_{k\ell} + \delta_{i,k-1}K_iX_{k-1,\ell} + (q\delta_{i+1,k}+q^{-1}\delta_{i,k})K_iF_iX_{k\ell},\\
	\text{and}\quad X_{k\ell}\blactright K_i &= (q^{-1}\delta_{i+1,k}+q\delta_{i,k})X_{k\ell}.
\end{align*}

From $\bar{h}\blactright c=c\actleft \bar{h}$, we have
\begin{align*}
	E_i \blactright X_{k\ell} &= (q\delta_{i,k,\ell}+q^{-1}\delta_{i-1,k,\ell})E_i+\delta_{i,k,\ell-1},\\
	F_i \blactright X_{k\ell} &= \delta_{k,\ell}F_i+\delta_{i,k-1,\ell}K_i^{-1},\\
	\text{and}\quad K_i \blactright X_{k\ell} &= (q\delta_{i,k,\ell}+q^{-1}\delta_{i+1,k,\ell})K_i.
\end{align*}

From $\bar{c}\blactright c=\counit(c)\bar{c}$, we have
\begin{equation*}
	X_{mn} \blactright X_{k\ell} = \delta_{k,\ell} X_{mn}.
\end{equation*}

In the next section we specialize to the case $H=\rU_q(\fsl_2)$.

\subsection{The actions of $\rD(\rU_q(\fsl_2))$ on $\rU_q(\fsl_2)\otimes\CC_q[\rSL_2]$}\label{section:duqsl2_actions}
Here we specialize our example from Section \ref{section:duqsln_actions} to $\rD(\rU_q(\fsl_2))$ acting on the algebra $\rU_q(\fsl_2)\otimes\CC_q[\rSL_2]$.

From $h\blactleft\bar{h}=\counit(h)\bar{h}$, we have
\begin{align*}
E\blactleft E &= 0 \qquad& E\blactleft F &= 0 \qquad& E\blactleft K &= 0 \qquad& E\blactleft K^{-1} &= 0\\
F\blactleft E &= 0 \qquad& F\blactleft F &= 0 \qquad& F\blactleft K &= 0 \qquad& F\blactleft K^{-1} &= 0\\
K\blactleft E &= E \quad& K\blactleft F &= F \quad& K\blactleft K &= K \quad& K\blactleft K^{-1} &= K^{-1}\\
K^{-1}\blactleft E &= E \quad& K^{-1}\blactleft F &= F \quad& K^{-1}\blactleft K &= K \quad& K^{-1}\blactleft K^{-1} &= K^{-1}
\end{align*}

From $h\blactleft\bar{c}=h\actleft \bar{c}$, we have
\begin{align*}
E\blactleft a &= 0 \qquad& E\blactleft b &= a \qquad& E\blactleft c &= 0 \qquad& E\blactleft d &= c\\
F\blactleft a &= b \qquad& F\blactleft b &= 0 \qquad& F\blactleft c &= d \qquad& F\blactleft d &= 0\\
K\blactleft a &= qa \quad& K\blactleft b &= q^{-1}b \quad& K\blactleft c &= qc \quad& K\blactleft d &= q^{-1}d\\
K^{-1}\blactleft a &= q^{-1}a \quad& K^{-1}\blactleft b &= qb \quad& K^{-1}\blactleft c &= q^{-1}c \quad& K^{-1}\blactleft d &= qd
\end{align*}

From $c\blactleft\bar{h}=(\bar{h}\actright c_{(2)})\cdot Sc_{(1)}c_{(3)}$, we have
\begin{alignat*}{4}
a&\blactleft E &&= E+qKcd			\qquad& a&\blactleft F &&= q^{-1}F-qba+(1-q^2)Fbc\\
a&\blactleft K &&= qK+(q^2-1)Kbc	\qquad& a&\blactleft K^{-1} &&= q^{-1}K^{-1}+(1-q^2)K^{-1}bc\\
b&\blactleft E &&= Kd^2				\qquad& b&\blactleft F &&= -qb^2+(1-q^2)Fbd\\
b&\blactleft K &&= (q^2-1)Kbd		\qquad& b&\blactleft K^{-1} &&= (1-q^2)K^{-1}bd\\
c&\blactleft E &&= -q^{-1}Kc^2		\qquad& c&\blactleft F &&= a^2+(q-q^{-1})Fac\\
c&\blactleft K &&= (q^{-1}-q)Kac	\qquad& c&\blactleft K^{-1} &&= (q-q^{-1})K^{-1}ac\\
d&\blactleft E &&= E-q^{-1}Kcd		\qquad& d&\blactleft F &&= qF+ab+(1-q^{-2})Fbc\\
d&\blactleft K &&= q^{-1}K+(q^{-2}-1)Kbc	\qquad& d&\blactleft K^{-1} &&= qK^{-1}+(1-q^{-2})K^{-1}bc
\end{alignat*}

From $c\blactleft\bar{c}=Sc_{(1)}\bar{c}c_{(2)}$, we have
\begin{alignat*}{4}
a&\blactleft a &&= a+(q-1)bca		\qquad& a&\blactleft b &&= qb+(q^2-q)b^2c\\
a&\blactleft c &&= qc+(q^2-q)bc^2	\qquad& a&\blactleft d &&= d+(q-1)dbc\\
b&\blactleft a &&= (1-q)b+(q-1)b^2c	\qquad& b&\blactleft b &&= (1-q^{-1})db^2\\
b&\blactleft c &&= (1-q^{-1})dcb	\qquad& b&\blactleft d &&= (1-q^{-1})d^2b\\
c&\blactleft a &&= (1-q)a^2c		\qquad& c&\blactleft b &&= (1-q)abc\\
c&\blactleft c &&= (1-q)ac^2		\qquad& c&\blactleft d &&= (1-q^{-1})c+(q^{-1}-1)bc^2\\
d&\blactleft a &&= a+(q^{-1}-1)abc	\qquad& d&\blactleft b &&= q^{-1}b+(q^{-2}-q^{-1})b^2c\\
d&\blactleft c &&= q^{-1}c+(q^{-2}-q^{-1})bc^2	\qquad& d&\blactleft d &&= d+(q^{-1}-1)bcd
\end{alignat*}

From $\bar{h}\blactright h=Sh_{(1)} \bar{h} h_{(2)}$, we have
\begin{alignat*}{4}
(1-q^{-2})E^2 &= E&&\blactright E
	\quad& (1-q^2)EF-\tfrac{K-K^{-1}}{q-q^{-1}} &= F&&\blactright E\\
(q^2-1)EK &= K&&\blactright E
	\quad& (q^{-2}-1)EK^{-1} &= K^{-1}&&\blactright E\\
K\cdot\tfrac{K-K^{-1}}{q-q^{-1}} &= E&&\blactright F
	\quad& 0 &= F&&\blactright F\\
(1-q^2)K^2F &= K&&\blactright F
	\quad& (1-q^{-2})F &= K^{-1}&&\blactright F\\
q^{-2}E &= E&&\blactright K
	\quad& q^2F &= F&&\blactright K\\
K &= K&&\blactright K
	\quad& K^{-1} &= K^{-1}&&\blactright K\\
q^2E &= E&&\blactright K^{-1}
	\quad& q^{-2}F &= F&&\blactright K^{-1}\\
K &= K&&\blactright K^{-1}
	\quad& K^{-1} &= K^{-1}&&\blactright K^{-1}
\end{alignat*}

From $\bar{c}\blactright h=Sh_{(1)} h_{(3)}\cdot (\bar{c}\actright h_{(2)})$, we have
\begin{alignat*}{4}
(1-q)Ea+Kc &= a&&\blactright E
	\quad& (1-q)Eb+Kd &= b&&\blactright E\\
(1-q^{-1})Ec &= c&&\blactright E
	\quad& (1-q^{-1})Ed &= c&&\blactright E\\
(q^{-1}-1)KFa &= a&&\blactright F
	\quad& (q^{-1}-1)KFb &= b&&\blactright F\\
(q-1)KFc+Ka &= c&&\blactright F
	\quad& (q-1)KFd+Kb &= c&&\blactright F\\
qa &= a&&\blactright K
	\quad& qb &= b&&\blactright K\\
q^{-1}c &= c&&\blactright K
	\quad& q^{-1}d &= c&&\blactright K\\
q^{-1}a &= a&&\blactright K^{-1}
	\quad& q^{-1}b &= b&&\blactright K^{-1}\\
qc &= c&&\blactright K^{-1}
	\quad& qd &= c&&\blactright K^{-1}
\end{alignat*}

From $\bar{h}\blactright c=c\actleft \bar{h}$, we have
\begin{alignat*}{8}
qE &= E&&\blactright a
	\qquad& F &= F&&\blactright a\qquad&
qK &= K&&\blactright a
	\qquad& q^{-1}K^{-1} &= K^{-1}&&\blactright a\\
1 &= E&&\blactright b
	\qquad& 0 &= F&&\blactright b\qquad&
0 &= K&&\blactright b
	\qquad& 0 &= K^{-1}&&\blactright b\\
0 &= E&&\blactright c
	\qquad& K^{-1} &= F&&\blactright c\qquad&
0 &= K&&\blactright c
	\qquad& 0 &= K^{-1}&&\blactright c\\
q^{-1}E &= E&&\blactright d
	\qquad& F &= F&&\blactright d\qquad&
q^{-1}K &= K&&\blactright d
	\qquad& qK^{-1} &= K^{-1}&&\blactright d
\end{alignat*}

From $\bar{c}\blactright c=\counit(c)\bar{c}$, we have
\begin{alignat*}{8}
a &= a&&\blactright a	\qquad& b &= b&&\blactright a\qquad&
c &= c&&\blactright a	\qquad& d &= d&&\blactright a\\
0 &= a&&\blactright b	\qquad& 0 &= b&&\blactright b\qquad&
0 &= c&&\blactright b	\qquad& 0 &= d&&\blactright b\\
0 &= a&&\blactright c	\qquad& 0 &= b&&\blactright c\qquad&
0 &= c&&\blactright c	\qquad& 0 &= d&&\blactright c\\
a &= a&&\blactright d	\qquad& b &= b&&\blactright d\qquad&
c &= c&&\blactright d	\qquad& c &= d&&\blactright d
\end{alignat*}

\subsection{Some simple sub-bimodules of $(\rU_q(\fsl_2)\otimes \CC_q[\rSL_2])_f$}\label{section:DH_bimodule_examples}

In this section we exhibit three simple sub-bimodules of $(\dualdouble)_f$.  In each example, we examine the subspace of highest-weight bivectors---those vectors annihilated by both the left and right actions of $E$.

The first example, $\cH_{1,1}$, is 16-dimensional and has a 4-dimensional subspace of highest-weight bivectors.  Its tensor square is 100-dimensional and is the internal direct sum of four non-isomorphic simple sub-bimodules of dimensions 81, 9, 9, and 1, respectively:
\[ \cH_{1,1}\otimes \cH_{1,1} \cong \cH_{2,2}\oplus \cH_{2,0}\oplus \cH_{0,2}\oplus \cH_{0,0} \]
We will exhibit the 9-dimensional bimodules $\cH_{2,0}$ and $\cH_{0,2}$, each of which has a one-dimensional subspace of highest-weight bivectors.  Of course, $\cH_{0,0}=\lr{1}$ is the trivial bimodule.

\begin{example}\label{ex:h11}
Let $H=\rU_q(\fsl_2)$ and $C=\CC_q[\rSL_2]$.  
Then $\cH_{1,1}$ is the 16-dimensional $\rD(H)$-sub-bimodule of 
$(H\otimes C)_f$ with basis
\begin{align*}
v_{11}&=EK^{-1}\\
v_{12}&=K^{-1}\\
v_{13}&=F\\
v_{14}&=\Delta\\
v_{21}&=(q-q^{-1})EK^{-1}ac-c^2\\
v_{22}&=(q-q^{-1})K^{-1}ac\\
v_{23}&=(q-q^{-1})Fac+a^2\\
v_{24}&=(q-q^{-1})\Delta ac-\textstyle\frac{q+q^{-1}}{q-q^{-1}}Kac-q^{-2}FKc^2+Ea^2\\
v_{31}&=(q^{-1}-q)EK^{-1}db+qd^2\\
v_{32}&=(q^{-1}-q)K^{-1}db\\
v_{33}&=(q^{-1}-q)Fdb-qb^2\\
v_{34}&=(q^{-1}-q)\Delta db+\textstyle\frac{q+q^{-1}}{q-q^{-1}}Kdb+q^{-1}FKd^2-qEb^2\\
v_{41}&=(q-q^{-1})EK^{-1}bc-dc\\
v_{42}&=(q-q^{-1})K^{-1}bc\\
v_{43}&=(q-q^{-1})Fbc+qab\\
v_{44}&=(q-q^{-1})\Delta bc-\textstyle\frac{q+q^{-1}}{q-q^{-1}}Kbc - q^{-2}FKdc + qEab - \textstyle \frac{1}{q-q^{-1}}K
\end{align*}

As a left $\rD(H)$-module, $\cH_{1,1}=V_1\oplus V_2\oplus V_3\oplus V_4$, where each $V_j$ is the left $\rD(H)$-module with basis $\{v_{1j},v_{2j},v_{3j},v_{4j}\}$.  In this basis, the left action of $x\in \rD(H)$ on $v\in V_j$ is given by $x\blactleft v=\phi_j(x)v$, where
\begin{align*}
\phi_j(E)&=\begin{matrixA}0 & 0 & \makebox[0em][c]{$q^{-1}-q$} & 0 \\0 & 0 & 0 & q \\0 & 0 & 0 & 0 \\0 & 0 & \makebox[0em][c]{$-q^{-1}-q$} & 0\end{matrixA}&
\phi_j(F)&=\begin{matrixB}0 & \makebox[0em][c]{$1-q^{-2}$} & 0 & 0 \\0 & 0 & 0 & 0 \\0 & 0 & 0 & \makebox[0em][c]{$-1$} \\0 & \makebox[0em][c]{$1+q^{-2}$} & 0 & 0\end{matrixB}\\
\phi_j(K)&=\begin{matrixA}1 & 0 & 0 & 0 \\0 & \makebox[0em][c]{$q^2$} & 0 & 0 \\0 & 0 & q^{-2} & 0 \\0 & 0 & 0 & 1\end{matrixA}&
\phi_j(K^{-1})&=\begin{matrixB}1 & 0 & 0 & 0 \\0 & q^{-2} & 0 & 0 \\0 & 0 & q^2 & 0 \\0 & 0 & 0 & 1\end{matrixB}\\
\phi_j(a)&=\begin{matrixA}\makebox[0em][c]{$q^{-1}$} & 0 & 0 & 0 \\0 & 1 & 0 & 0 \\0 & 0 & 1 & 0 \\\makebox[0em][c]{$-q$} & 0 & 0 & q\end{matrixA}&
\phi_j(b)&=\begin{matrixB}0 & 0 & 0 & 0 \\0 & 0 & 0 & 0 \\1 & 0 & 0 & 0 \\0 & \makebox[0em][c]{$q^{-1}-q$} & 0 & 0\end{matrixB}\\
\phi_j(c)&=\begin{matrixA}0 & 0 & 0 & 0 \\1 & 0 & 0 & 0 \\0 & 0 & 0 & 0 \\0 & 0 & q^{-1}-q & 0\end{matrixA}&
\phi_j(d)&=\begin{matrixB}q & 0 & 0 & 0 \\0 & 1 & 0 & 0 \\0 & 0 & 1 & 0 \\\makebox[0em][c]{$q^{-1}$} & 0 & 0 & \makebox[0em][c]{$q^{-1}$}\end{matrixB}
\end{align*}
For example, $d\blactleft v_{1j}=qv_{1j}+q^{-1}v_{4j}$ for $j=1,2,3,4$.

As a right $\rD(H)$-module, $\cH_{1,1}=V_1'\oplus V_2'\oplus V_3'\oplus V_4'$, where each $V_i'$ is the right $\rD(H)$-module with basis $\{v_{i1},v_{i2},v_{i3},v_{i4}\}$.  In this basis, the right action of $x\in \rD(H)$ on $v\in V_i'$ is given by $v\blactright x=\phi_i'(x)v$, where
\begin{align*}
\phi_i'(E)&=\begin{matrixC}0 & \makebox[0em][c]{$q^{-2}-1$} & 0 & 0 \\0 & 0 & \makebox[0em][c]{$\frac{q^2+1}{q-q^{-1}}$} & 0 \\0 & 0 & 0 & 0 \\0 & 0 & \makebox[0em][c]{$1-q^2$} & 0\end{matrixC}&
\phi_i'(F)&=\begin{matrixD}0 & 0 & 0 & 0 \\\makebox[0em][c]{$\frac{q^2+1}{q^{-1}-q}$} & 0 & 0 & 0 \\0 & \makebox[0em][c]{$1-q^{-2}$} & 0 & 0 \\\makebox[0em][c]{$q^2-1$} & 0 & 0 & 0\end{matrixD}\\
\phi_i'(K)&=\begin{matrixC}\makebox[0em][c]{$q^{-2}$} & 0 & 0 & 0 \\0 & 1 & 0 & 0 \\0 & 0 & q^2 & 0 \\0 & 0 & 0 & 1\end{matrixC}&
\phi_i'(K^{-1})&=\begin{matrixD}q^2 & 0 & 0 & 0 \\0 & 1 & 0 & 0 \\0 & 0 & \makebox[0em][c]{$q^{-2}$} & 0 \\0 & 0 & 0 & 1\end{matrixD}\\
\phi_i'(a)&=\begin{matrixC}1 & 0 & 0 & 0 \\0 & q^{-1} & 0 & \makebox[0em][c]{$\frac{q^{-1}}{q^{-1}-q}$} \\0 & 0 & 1 & 0 \\0 & 0 & 0 & q\end{matrixC}&
\phi_i'(b)&=\begin{matrixD}0 & 0 & 0 & 0 \\q & 0 & 0 & 0 \\0 & 0 & 0 & 1 \\0 & 0 & 0 & 0\end{matrixD}\\
\phi_i'(c)&=\begin{matrixC}0 & 0 & 0 & \makebox[0em][c]{$q^{-1}$} \\0 & 0 & 1 & 0 \\0 & 0 & 0 & 0 \\0 & 0 & 0 & 0\end{matrixC}&
\phi_i'(d)&=\begin{matrixD}1 & 0 & 0 & 0 \\0 & q & 0 & \makebox[0em][c]{$\frac{q}{q-q^{-1}}$} \\0 & 0 & 1 & 0 \\0 & 0 & 0 & \makebox[0em][c]{$q^{-1}$}\end{matrixD}
\end{align*}
For example, $v_{i4}\blactright c = q^{-1}v_{i1}$ for $i=1,2,3,4$.

We see that $\phi_j(E)$ is a rank-2 matrix with null space spanned by $\{v_{1j},v_{2j}\}$ and $\phi_i'(E)$ is a rank-2 matrix with null space spanned by $\{v_{i1},v_{i4}\}$.  It follows that the subspace $\{v\in \cH_{1,1}\mid E\blactleft v=v\blactright E=0\}$ of highest-weight bivectors is spanned by $\{v_{11},v_{14},v_{21},v_{24}\}$.
We note that $v_{21}$ has weight $(2,2)$, $v_{24}$ has weight $(2,0)$, $v_{11}$ has weight $(0,2)$, and $v_{14}$ has weight $(0,0)$.  Therefore, the basis $\{v_{21},v_{24},v_{11},v_{14}\}$ is canonical up to scaling.  Elsewhere in this paper we refer to these four highest-weight bivectors as
\begin{subequations}\label{eqn:v_1v_2v_3v_4}
\begin{align}
	\bfv_1&=v_{21}=(q-q^{-1})EK^{-1}ac-c^2\\
	\bfv_2&=v_{24}=(q-q^{-1})\Delta ac-\tfrac{q+q^{-1}}{q-q^{-1}}Kac-q^{-2}FKc^2+Ea^2\\
	\bfv_3&=v_{11}=EK^{-1}\\
	\bfv_4&=v_{14}=\Delta
\end{align}
\end{subequations}
\end{example}

\begin{example}\label{ex:h20}
Let $H=\rU_q(\fsl_2)$ and $C=\CC_q[\rSL_2]$.
Then $\cH_{2,0}\subset \cH_{1,1}\otimes \cH_{1,1}$ is the 9-dimensional $\rD(H)$-sub-bimodule of $(H\otimes C)_f$ with basis
\begin{align*}
\vd_{11}&=K^{-1}c^2\\
\vd_{12}&=(q-q^{-1})Fc^{2} + qac\\
\vd_{13}&=q^{-3}(q-q^{-1})^2F^{2}Kc^{2} + q^{-1}(q^2-q^{-2})FKac + Ka^{2}\\
\vd_{21}&=q^{-1}K^{-1}dc\\
\vd_{22}&=q^{-1}(q-q^{-1})Fdc + bc + \tfrac{1}{q + q^{-1}}\\
\vd_{23}&=q^{-4}(q-q^{-1})^2F^{2}Kdc + q^{-2}(q^2-q^{-2})FKbc + q^{-2}(q-q^{-1})FK + Kab\\
\vd_{31}&=K^{-1}d^2\\
\vd_{32}&=(q-q^{-1})Fd^{2} + db\\
\vd_{33}&=q^{-3}(q-q^{-1})^2F^{2}Kd^{2} + q^{-2}(q^2-q^{-2})FKdb + Kb^{2}
\end{align*}

As a left $\rD(H)$-module, $\cH_{2,0}=V_1\oplus V_2\oplus V_3$, where each $V_j$ is the left $\rD(H)$-module with basis $\{\vd_{1j},\vd_{2j},\vd_{3j}\}$.  In this basis, the left action of $x\in \rD(H)$ on $v\in V_j$ is given by $x\blactleft v=\phi_j(x)v$, where
\begin{align*}
\phi_j(E)&=\begin{matrixE}0 & 1 & 0 \\0 & 0 & \makebox[0em][c]{$q+q^{-1}$} \\0 & 0 & 0\end{matrixE}&
\phi_j(F)&=\begin{matrixF}0 & 0 & 0 \\\makebox[0em][c]{$q+q^{-1}$} & 0 & 0 \\0 & 1 & 0\end{matrixF}\\
\phi_j(K)&=\begin{matrixE}\makebox[0em][c]{$q^2$} & 0 & 0\\0 & 1 & 0\\0 & 0 & q^{-2}\end{matrixE}&
\phi_j(K^{-1})&=\begin{matrixF}\makebox[0em][c]{$q^{-2}$} & 0 & 0 \\0 & 1 & 0 \\0 & 0 & q^2\end{matrixF}\\
\phi_j(a)&=\begin{matrixE}q & 0 & 0 \\0 & 1 & 0 \\0 & 0 & q^{-1}\end{matrixE}&
\phi_j(b)&=\begin{matrixF}0 & 0 & 0 \\0 & 0 & 0 \\0 & 0 & 0\end{matrixF}\\
\phi_j(c)&=\begin{matrixE}0 & \makebox[0em][c]{$1-q^{-2}$} & 0 \\0 & 0 & \makebox[0em][c]{$q^2-q^{-2}$} \\0 & 0 & 0\end{matrixE}&
\phi_j(d)&=\begin{matrixF}q^{-1} & 0 & 0 \\0 & 1 & 0 \\0 & 0 & q\end{matrixF}
\end{align*}

As a right $\rD(H)$-module, $\cH_{2,0}=V_1'\oplus V_2'\oplus V_3'$, where each $V_i'$ is the right $\rD(H)$-module with basis $\{\vd_{i1},\vd_{i2},\vd_{i3}\}$.  In this basis, the right action of $x\in \rD(H)$ on $v\in V_i'$ is given by $v\blactright x=\phi_i'(x)v$, where
\begin{align*}
\phi_i'(E)&=\begin{matrixE}0 & 1 & 0 \\0 & 0 & \makebox[0em][c]{$q+q^{-1}$} \\0 & 0 & 0\end{matrixE}&
\phi_i'(F)&=\begin{matrixF}0 & 0 & 0 \\\makebox[0em][c]{$q+q^{-1}$} & 0 & 0 \\0 & 1 & 0\end{matrixF}\\
\phi_i'(K)&=\begin{matrixE}\makebox[0em][c]{$q^{-2}$} & 0 & 0 \\0 & 1 & 0 \\0 & 0 & q^2\end{matrixE}&
\phi_i'(K^{-1})&=\begin{matrixF}q^2 & 0 & 0 \\0 & 1 & 0 \\0 & 0 & \makebox[0em][c]{$q^{-2}$}\end{matrixF}\\
\phi_i'(a)&=\begin{matrixE}\makebox[0em][c]{$q^{-1}$} & 0 & 0 \\0 & 1 & 0 \\0 & 0 & q\end{matrixE}&
\phi_i'(b)&=\begin{matrixF}0 & 0 & 0 \\0 & 0 & 0 \\0 & 0 & 0\end{matrixF}\\
\phi_i'(c)&=\begin{matrixE}0 & \makebox[0em][c]{$q-q^{-1}$} & 0 \\0 & 0 & \makebox[0em][c]{$q-q^{-3}$} \\0 & 0 & 0\end{matrixE}&
\phi_i'(d)&=\begin{matrixF}q & 0 & 0 \\0 & 1 & 0 \\0 & 0 & \makebox[0em][c]{$q^{-1}$}\end{matrixF}
\end{align*}

We see that $\phi_j(E)$ is a rank-2 matrix with null space spanned by $\{\vd_{1j}\}$ and $\phi_i'(E)$ is a rank-2 matrix with null space spanned by $\{\vd_{i1}\}$.  It follows that the subspace $\{v\in \cH_{2,0}\mid E\blactleft v=v\blactright E=0\}$ of highest-weight bivectors is spanned by $\vd_{11}$.  Elsewhere in this paper we refer to this bivector as
\begin{align}\label{eqn:v_5}
	\bfv_5&=\vd_{11}=K^{-1}c^2,
\end{align}
and we note that $\frac{q^2+1}{q-q^{-1}}\,\bfv_5 = (1-q^2)\bfv_1\bfv_4-[\bfv_2,\bfv_3]_{q^2}$.
\end{example}

\begin{example}\label{ex:h02}
Let $H=\rU_q(\fsl_2)$ and $C=\CC_q[\rSL_2]$.
Then $\cH_{0,2}\subset \cH_{1,1}\otimes \cH_{1,1}$ is the $\rD(H)$-sub-bimodule of $(H\otimes C)_f$ with basis
\begin{align*}
	\vdd_{11}&=q^{-1}(q-q^{-1})^2E^2K^{-1}a^2-q(q^2-q^{-2})Eac+Kc^2\\
	\vdd_{12}&=(q^{-1}-q)EK^{-1}a^2+qac\\
	\vdd_{13}&=K^{-1}a^2\\
	\vdd_{21}&=q^{-1}(q-q^{-1})^2E^2K^{-1}ab-(q^2-q^{-2})Ebc-(q-q^{-1})E+q^{-1}Kdc\\
	\vdd_{22}&=(q^{-1}-q)EK^{-1}ab+bc+\tfrac{1}{q+q^{-1}}\\
	\vdd_{23}&=K^{-1}ab\\
	\vdd_{31}&=q^{-1}(q-q^{-1})^2E^2K^{-1}b^2-(q^2-q^{-2})Edb+Kd^2\\
	\vdd_{32}&=(q^{-1}-q)EK^{-1}b^2+db\\
	\vdd_{33}&=K^{-1}b^2
\end{align*}

As a left $\rD(H)$-module, $\cH_{0,2}=V_1\oplus V_2\oplus V_3$, where each $V_j$ is the left $\rD(H)$-module with basis $\{\vdd_{1j},\vdd_{2j},\vdd_{3j}\}$.  In this basis, the left action of $x\in \rD(H)$ on $v\in V_j$ is given by $x\blactleft v=\phi_j(x)v$, where
\begin{align*}
\phi_j(E)&=\begin{matrixG}0 & 1 & 0 \\0 & 0 & \makebox[0em][c]{$q+q^{-1}$} \\0 & 0 & 0\end{matrixG}&
\phi_j(F)&=\begin{matrixH}0 & 0 & 0 \\\makebox[0em][c]{$q+q^{-1}$} & 0 & 0 \\0 & 1 & 0\end{matrixH}\\
\phi_j(K)&=\begin{matrixG}\makebox[0em][c]{$q^2$} & 0 & 0\\0 & 1 & 0\\0 & 0 & q^{-2}\end{matrixG}&
\phi_j(K^{-1})&=\begin{matrixH}\makebox[0em][c]{$q^{-2}$} & 0 & 0 \\0 & 1 & 0 \\0 & 0 & \makebox[0em][c]{$q^2$}\end{matrixH}\\
\phi_j(a)&=\begin{matrixG}\makebox[0em][c]{$q^{-1}$} & 0 & 0 \\0 & 1 & 0 \\0 & 0 & q\end{matrixG}&
\phi_j(b)&=\begin{matrixH}0 & 0 & 0 \\\makebox[0em][c]{$q^{-1}-q^3$} & 0 & 0 \\0 & \makebox[0em][c]{$q^{-1}-q$} & 0\end{matrixH}\\
\phi_j(c)&=\begin{matrixG}0 & 0 & 0 \\0 & 0 & 0 \\0 & 0 & 0\end{matrixG}&
\phi_j(d)&=\begin{matrixH}q & 0 & 0 \\0 & 1 & 0 \\0 & 0 & \makebox[0em][c]{$q^{-1}$}\end{matrixH}
\end{align*}

As a right $\rD(H)$-module, $\cH_{0,2}=V_1'\oplus V_2'\oplus V_3'$, where each $V_i'$ is the right $\rD(H)$-module with basis $\{\vdd_{i1},\vdd_{i2},\vdd_{i3}\}$.  In this basis, the right action of $x\in \rD(H)$ on $v\in V_i'$ is given by $v\blactright x=\phi_i'(x)v$, where
\begin{align*}
\phi_i'(E)&=\begin{matrixG}0 & 1 & 0 \\0 & 0 & \makebox[0em][c]{$q+q^{-1}$} \\0 & 0 & 0\end{matrixG}&
\phi_i'(F)&=\begin{matrixH}0 & 0 & 0 \\\makebox[0em][c]{$q+q^{-1}$} & 0 & 0 \\0 & 1 & 0\end{matrixH}\\
\phi_i'(K)&=\begin{matrixG}\makebox[0em][c]{$q^{-2}$} & 0 & 0 \\0 & 1 & 0 \\0 & 0 & \makebox[0em][c]{$q^2$}\end{matrixG}&
\phi_i'(K^{-1})&=\begin{matrixH}q^2 & 0 & 0 \\0 & 1 & 0 \\0 & 0 & \makebox[0em][c]{$q^{-2}$}\end{matrixH}\\
\phi_i'(a)&=\begin{matrixG}q & 0 & 0 \\0 & 1 & 0 \\0 & 0 & \makebox[0em][c]{$q^{-1}$}\end{matrixG}&
\phi_i'(b)&=\begin{matrixH}0 & 0 & 0 \\\makebox[0em][c]{$q^{-2}-q^{2}$} & 0 & 0 \\0 & \makebox[0em][c]{$1-q^{2}$} & 0\end{matrixH}\\
\phi_i'(c)&=\begin{matrixG}0 & 0 & 0 \\0 & 0 & 0 \\0 & 0 & 0\end{matrixG}&
\phi_i'(d)&=\begin{matrixH}\makebox[0em][c]{$q^{-1}$} & 0 & 0 \\0 & 1 & 0 \\0 & 0 & q\end{matrixH}
\end{align*}

We see that $\phi_j(E)$ is a rank-2 matrix with null space spanned by $\{\vdd_{1j}\}$ and $\phi_i'(E)$ is a rank-2 matrix with null space spanned by $\{\vdd_{i1}\}$.  It follows that the subspace $\{v\in \cH_{0,2}\mid E\blactleft v=v\blactright E=0\}$ of highest-weight bivectors is spanned by $\vdd_{11}$.  Elsewhere in this paper we refer to this bivector as
\begin{align}\label{eqn:v_6}
	\bfv_6&=\vdd_{11}=q^{-1}(q-q^{-1})^2E^2K^{-1}a^2-q(q^2-q^{-2})Eac+Kc^2,
\end{align}
and we note that $\frac{q}{q^2-q^{-2}}\,\bfv_6 = (1-q^2)\bfv_1\bfv_4-[\bfv_3,\bfv_2]_{q^2}$.
\end{example}

\begin{remark}
We see that $\cH_{2,0}$ is not isomorphic to $\cH_{0,2}$ because $b$ annihilates $\cH_{2,0}$ but does not annihilate $\cH_{0,2}$.
\end{remark}

\begin{conjecture}
Based on these examples, we conjecture that $\cH_{\lambda,\mu}$ is generated as a $\rD(H)$-bimodule by $h_{\lambda,\mu}$ where
\begin{equation*}
	h_{\lambda,\mu} =
	\begin{cases}
		K^{-(\lambda+\mu)/2}c^{\lambda-\mu} &\text{if $\lambda\geq\mu$,}\\
		K^{-(\lambda+\mu)/2}b^{\mu-\lambda} &\text{if $\lambda<\mu$.}
	\end{cases}
\end{equation*}
\end{conjecture}

\section{Proofs}\label{C:Proofs}

\subsection{Proof of Theorem \ref{thm:dual_is_a_bimodule}}

We define a pairing $\lr{\;,\,}\colon \rD(H)\otimes(H\otimes C)\to k$ by
\begin{equation}\label{eqn:double_pairing}
	\lr{ c\cdot h,\bar{h}\cdot \bar{c} } = \phi(c,\bar{h})\phi(\bar{c},h).
\end{equation}
As described in Section \ref{section:hopf_pairing}, we can use this pairing to define left and right actions of $\rD(H)$ on $H\otimes C$.  The rest of the proof is a direct calculation.

We define $\blactright\colon (H\otimes C)\otimes \rD(H)\to H\otimes C$ so that
\begin{equation*}
	\begin{split}
		\langle c'\cdot h'&,(\bar{h}\cdot \bar{c}) \blactright (c\cdot h)\rangle\\
		&=\langle (c\cdot h)\cdot(c'\cdot h'),\bar{h}\cdot \bar{c}\rangle\\
		&=\langle c\cdot c'_{(2)}\cdot h_{(2)}\cdot h',\bar{h}\cdot \bar{c}\rangle
			\phi(c'_{(1)},Sh_{(1)})\phi(c'_{(3)},h_{(3)})\\
		&=\phi(c\cdot c'_{(2)},\bar{h})\phi(\bar{c},h_{(2)}\cdot h')
			\phi(c'_{(1)},Sh_{(1)})\phi(c'_{(3)},h_{(3)})\\
		&=\phi(c'_{(2)}c,\bar{h})\phi(\bar{c},h_{(2)} h')
			\phi(c'_{(1)},Sh_{(1)})\phi(c'_{(3)},h_{(3)})\\
		&=\phi(c'_{(2)},\bar{h}_{(1)})\phi(c,\bar{h}_{(2)})\phi(\bar{c}_{(1)},h_{(2)})\phi(\bar{c}_{(2)},h')
			\phi(c'_{(1)},Sh_{(1)})\phi(c'_{(3)},h_{(3)})\\
		&=\phi(c',Sh_{(1)}\bar{h}_{(1)}h_{(3)})\phi(c,\bar{h}_{(2)})\phi(\bar{c}_{(1)},h_{(2)})\phi(\bar{c}_{(2)},h')\\
		&=\langle c'\cdot h',Sh_{(1)}\bar{h}_{(1)}h_{(3)}\cdot \bar{c}_{(2)}\rangle\phi(c,\bar{h}_{(2)})\phi(\bar{c}_{(1)},h_{(2)})\\
		&=\langle c'\cdot h',Sh_{(1)}(c\actleft \bar{h})h_{(3)}\cdot(\bar{c}\actright h_{(2)})\rangle
	\end{split}
\end{equation*}
and thus
\begin{align*}
	(\bar{h}\cdot \bar{c})\blactright c&=(c\actleft \bar{h})\cdot \bar{c}\qquad\text{and}\qquad
	(\bar{h}\cdot \bar{c})\blactright h=Sh_{(1)} \bar{h} h_{(3)}\cdot (\bar{c}\actright h_{(2)}).
\end{align*}
We define $\blactleft\colon \rD(H)\otimes(H\otimes C)\to H\otimes C$ so that
\begin{equation*}
	\begin{split}
		\langle c'\cdot h',h \blactleft (\bar{h}\cdot \bar{c})\rangle
		&=\langle (c'\cdot h')(h),\bar{h}\cdot \bar{c}\rangle\\
		&=\langle c'\cdot h'h,\bar{h}\cdot \bar{c}\rangle\\
		&=\phi(c',\bar{h})\phi(\bar{c},h'h)\\
		&=\phi(c',\bar{h})\phi(\bar{c}_{(1)},h')\phi(\bar{c}_{(2)},h)\\
		&=\langle c'\cdot h',\bar{h}\cdot \bar{c}_{(1)}\rangle\phi(\bar{c}_{(2)},h)\\
		&=\langle c'\cdot h',\bar{h}\cdot (h\actleft \bar{c})\rangle
	\end{split}
\end{equation*}
and
\begin{equation*}
	\begin{split}
		\langle c'\cdot h',c \blactleft (\bar{h}\cdot \bar{c})\rangle
		&=\langle (c'\cdot h')(c),\bar{h}\cdot \bar{c}\rangle\\
		&=\langle c'\cdot c_{(2)}\cdot h'_{(2)},\bar{h}\cdot \bar{c}\rangle
			\phi(c_{(1)},Sh'_{(1)})\phi(c_{(3)},h'_{(3)})\\
		&=\phi(c_{(2)}c',\bar{h})\phi(\bar{c},h'_{(2)})
			\phi(c_{(1)},Sh'_{(1)})\phi(c_{(3)},h'_{(3)})\\
		&=\phi(c_{(2)},\bar{h}_{(1)})\phi(c',\bar{h}_{(2)})\phi(\bar{c},h'_{(2)})
			\phi(c_{(1)},Sh'_{(1)})\phi(c_{(3)},h'_{(3)})\\
		&=\phi(c_{(2)},\bar{h}_{(1)})\phi(c',\bar{h}_{(2)})\phi(\bar{c},h'_{(2)})
			\phi(Sc_{(1)},h'_{(1)})\phi(c_{(3)},h'_{(3)})\\
		&=\phi(c_{(2)},\bar{h}_{(1)})\phi(c',\bar{h}_{(2)})\phi(Sc_{(1)}\bar{c}c_{(3)},h')\\
		&=\phi(c_{(2)},\bar{h}_{(1)})\langle(c'\cdot h',\bar{h}_{(2)}\cdot(Sc_{(1)}\bar{c}c_{(3)})\rangle\\
		&=\langle(c'\cdot h',(\bar{h}\actright c_{(2)})\cdot Sc_{(1)}\bar{c}c_{(3)}\rangle.
	\end{split}
\end{equation*}
Thus
\begin{align*}
	c \blactleft (\bar{h}\cdot \bar{c})
		&=(\bar{h}\actright c_{(2)})\cdot Sc_{(1)}\bar{c}c_{(3)}\qquad\text{and}\qquad
	h \blactleft (\bar{h}\cdot \bar{c})
		=\bar{h}\cdot(h\actleft \bar{c})
\end{align*}
and the theorem is proved.\qed

\subsection{Proof of Theorem \ref{thm:v1v2v3v4generate_hw}}\label{section:prove_main_thm}

We recall the vectors from (\ref{eqn:v_1v_2v_3v_4}), (\ref{eqn:v_5}), and (\ref{eqn:v_6}):
\begin{align*}
	\bfv_1&=(q-q^{-1})EK^{-1}ac-c^2\\
	\bfv_2&=(q-q^{-1})\Delta ac-\tfrac{q+q^{-1}}{q-q^{-1}}Kac-q^{-2}FKc^2+Ea^2\\
	\bfv_3&=EK^{-1}\\
	\bfv_4&=\Delta\\
	\bfv_5&=K^{-1}c^2\\
	\bfv_6&=q^{-1}(q-q^{-1})^2E^2K^{-1}a^2-q(q^2-q^{-2})Eac+Kc^2
\end{align*}

Our goal is to show that $\HCpp$ has basis
\begin{align*}
	\{\bfv_3^\ell \bfv_1^m \bfv_5^p \bfv_2^s \}
	\cup\{\bfv_3^\ell \bfv_1^m \bfv_6^r \bfv_2^s \}
	\cup\{\bfv_3^\ell \bfv_4^n \bfv_5^p \bfv_2^s \}
	\cup\{\bfv_3^\ell \bfv_4^n \bfv_6^r \bfv_2^s \}.
\end{align*}
The biggest challenge is that $E$ and $F$ do not have a convenient commuting relation.  By localizing at $\bfv_3=EK^{-1}$, we will be able to write $F$ in terms of other vectors with much better commuting relations (See Lemma \ref{lem:noF}).  We will do the same for $a$ by localizing at $c$, and we will also show that neither $b$ nor $d$ appears in any highest-weight bivector. 
In summary, we will use the (not yet justified) embeddings
\begin{align}\label{eqn:localizations}
	\HCpp\subset{^+\cA[\bfv_3^{-1}]^+}\subset{^+\cA[\bfv_3^{-1},c^{-1}]^+}
\end{align} 
where $\cA\subset H\otimes C$ is the subalgebra with basis
\[
	\big\{ \bfv_3^m\bfv_4^n\bfv_5^pK^ka^\ell c^\epsilon \mid
	\text{$m,n,p,\ell\in\ZZ_{\geq 0}$, $k\in\ZZ$, $\epsilon\in\{0,1\}$}\big\}.
\]
It will be relatively easy to find all highest-weight bivectors in $\cA[\bfv_3^{-1},c^{-1}]$ (See Proposition \ref{prop:hw_bivectors_in_fully_localized_algebra}). We will then carefully determine which of these highest-weight bivectors can be written without $c^{-1}$ or $\bfv_3^{-1}$, and this will give us our basis for $\HCpp$.

First we must confirm that $\bfv_3$ and $c$ are Ore elements in $\cA$. Using Proposition \ref{prop:h_otimes_c} (Recall also Definition \ref{def:quantum_enveloping_algebra} and Example \ref{ex:coordinate_algebra}) and notation $[x,y]_q=xy-qyx$, we find the following relations: 
\begin{lemma}\label{lem:commutators}
The vector $\bfv_4$ is central in $H\otimes C$, and
\begin{align*}
	0
	=[\bfv_3,\bfv_5]_{q^2}
	=[\bfv_3,K]_{q^{-2}}
	=[\bfv_3,a]
	=[\bfv_3,c]
	=[\bfv_5,K]
	=[\bfv_5,a]_{q^2}
	=[\bfv_5,c].
\end{align*}
\end{lemma}

\begin{corollary}
Both $\bfv_3$ and $c$ are Ore elements in $\cA$.
\end{corollary}

Now we must also confirm that $\HCpp\subset{^+\cA[\bfv_3^{-1}]^+}$.  This is not obvious, because $\cA[\bfv_3^{-1}]$ appears to be missing $b$, $d$, and $F$.  Certainly $H\otimes C\not\subset \cA[\bfv_3^{-1}]$.  We will resolve this by determining that $\HCpp$ doesn't have $b$ or $d$ either (Lemma~\ref{lem:cplus}), and that $F\in\cA[\bfv_3^{-1}]$ can be written in terms of other vectors (Lemma~\ref{lem:noF}).

\begin{lemma}\label{lem:cplus}
The subspace $\Cplus=\{v\in C \mid E\blactleft v=0\}$ is spanned by $\{a^ic^j\mid i,j\geq 0\}$.
Furthermore, if $v\in \cH$, then $E\blactleft v=0$ if and only if $v\in H\otimes\Cplus$. That is, $^+(H\otimes C)\subset H\otimes{^+C}$.
\end{lemma}

\begin{proof}
Recall that $H\otimes C$ has basis $\{E^i F^j \Delta^k K^{\pm\ell} a^p c^r b^m d^n\mid ij=pn=0\}$.  Because
\[ E\blactleft E^i F^j \Delta^k K^\ell a^p c^r b^m d^n = E^i F^j \Delta^k K^\ell a^p c^r (E\blactleft b^m d^n), \]
we must consider $E\blactleft b^m d^n$. 
Recall the notation $[m]_q=\frac{q^m-q^{-m}}{q-q^{-1}}$. 
Now $E\blactleft b^m=[m]_qab^{m-1}$, $E\blactleft d^n=[n]_qcd^{n-1}$, 
and $K\blactleft d^n=q^{-n}d^n$, so if $n>0$ then
\begin{equation*}
	\begin{split}
	E\blactleft b^md^n&=[m]_qq^{-n}ab^{m-1}d^n+[n]_qb^mcd^{n-1}\\
	&=[m]_qq^{1-m-n}b^{m-1}(1+q^{-1}bc)d^{n-1}+[n]_qcb^md^{n-1}\\
	&=[m]_qq^{1-m-n}b^{m-1}d^{n-1}+\tfrac{q^{2n}-q^{-2m}}{q-q^{-1}}q^{-n}cb^md^{n-1}.
	\end{split}
\end{equation*}
We observe that there is a grading \( H\otimes C = \bigoplus_{m=0}^\infty A_m \) where \[ A_m = \Span\{E^i F^j \Delta^k K^{\pm\ell} a^p c^r b^m d^n\mid \text{$i,j,k,\ell,n,p,r\geq 0$ and $ij=pn=0$}\} \] and we let $\pi_m\colon H\otimes C\to A_m$ denote the canonical projection.  We also note that $E\blactleft A_m \subset A_m$ for all $m\geq 0$.

Let $v=\textstyle\sum \alpha_{i,j,k,\ell,m,n,p,r}E^i F^j \Delta^k K^{\ell} a^p c^r b^m d^n$ be a nonzero vector such that $E\blactleft v=0$, and let $M=\max\{m\in\ZZ \mid \pi_m(v)\neq 0\}$.  Then
\begin{align*}
	\pi_M(E\blactleft v) = \sum_{\genfrac{}{}{0pt}{}{i,j,k,n,p,r\geq 0}{\ell\in\ZZ,\,ij=pn=0}}\frac{q^{2n}-q^{-2M}}{q-q^{-1}}q^{-n}\alpha_{i,j,k,\ell,M,n,p,r}E^iF^j\Delta^kK^{\ell}a^pc^{r+1}b^Md^{n-1}.
\end{align*}
For this to be zero, we can have $\alpha_{i,j,k,\ell,M,n,p,r}\neq0$ only if $M=n=0$.  Thus $v\in \Span\{E^i F^j \Delta^k K^{\ell} a^p c^r\}$.
\end{proof}

\begin{lemma}\label{lem:noF}
$H\otimes\Cplus$ is a subalgebra of $\cA[\bfv_3^{-1}]$.
\end{lemma}

\begin{proof}
We write $E=\bfv_3K$,
\begin{equation*}
	\begin{split}
		F
		&= (FE)(\bfv_3K)^{-1}\\
		&= \bigg(\bfv_4-\frac{qK+q^{-1}K^{-1}}{(q-q^{-1})^{2}}\bigg)q^2\bfv_3^{-1}K^{-1}\\
		&= q^2\bfv_3^{-1}\bfv_4K^{-1}-\frac{q}{(q-q^{-1})^{2}}\bfv_3^{-1}-\frac{q^{3}}{(q-q^{-1})^{2}}\bfv_3^{-1}K^{-2}
	\end{split}
\end{equation*}
and $c^{2p+\epsilon}=\bfv_5^pK^pc^\epsilon$.
\end{proof}

We have verified the statement \eqref{eqn:localizations}. We will now find a basis for $^+\cA[\bfv_3^{-1},c^{-1}]^+$ which we can use later to find a basis for $^+(H\otimes C)^+$.

\begin{proposition}\label{prop:hw_bivectors_in_fully_localized_algebra}
The algebra $^+\cA[\bfv_3^{-1},c^{-1}]^+$ is generated by $\{\bfv_1,\bfv_3^{\pm 1},\bfv_4,\bfv_5^{\pm 1}\}$.
\end{proposition}

\begin{proof}
Because $(q-q^{-1})a=\bfv_1\bfv_3^{-1}c^{-1}+\bfv_3^{-1}\bfv_5Kc^{-1}$, we know that
\[
	\big\{ \bfv_1^s\bfv_3^m\bfv_4^n\bfv_5^pK^kc^\epsilon \mid\text{$s,n\geq 0$ and $\epsilon\in\{0,1\}$}\big\}
\]
is a basis of $\cA[\bfv_3^{-1},c^{-1}]$.

We already know that $E\blactleft v=0$ for all $v\in \cA[\bfv_3^{-1},c^{-1}]$.  We therefore must show that the solutions to $v\blactright E=0$ are those vectors $v$ belonging to the subspace spanned by $\{\bfv_1^s\bfv_3^m\bfv_4^n\bfv_5^p\mid \text{$s,n\geq 0$}\}$.  Suppose that
\[
	0=\big(\textstyle\sum \bfv_1^s\bfv_3^m\bfv_4^n\bfv_5^pK^k
		(\alpha_{s,m,n,p,k}
		+c\beta_{s,m,n,p,k})
		\big)\blactright E
\]
where the $\alpha$'s and $\beta$'s are coefficients. 
We will show that $\beta_{s,m,n,p,k}=0$ for all indices and $\alpha_{s,m,n,p,k}=0$ whenever $k\neq 0$.  Now
\[
	K^k \blactright E
	= (q^{2k}-1)EK^k
	= (q^{2k}-1)\bfv_3K^{k+1}
\]
and
\begin{equation*}
	\begin{split}
		K^kc \blactright E
		&=(K^k\blactright E)(c\blactright K)+K^k(c\blactright E)\\
		&=(q^{2k}-1)EK^kq^{-1}c
			+K^k(1-q^{-1})Ec\\
		&= (q^{2k-1}-q^{-1})\bfv_3K^{k+1}c
			+(1-q^{-1})K^k\bfv_3Kc\\
		&= (q^{2k-1}-q^{-1})\bfv_3K^{k+1}c
			+(q^{2k}-q^{2k-1})\bfv_3K^{k+1}c\\
		&= (q^{2k}-q^{-1})\bfv_3K^{k+1}c
	\end{split}
\end{equation*}
so
\begin{equation*}
	\begin{split}
		0
		&=\big(\textstyle\sum \bfv_1^s\bfv_3^m\bfv_4^n\bfv_5^pK^k
		(\alpha_{s,m,n,p,k}
		+c\beta_{s,m,n,p,k})
		\big)\blactright E\\
		&=\textstyle\sum \bfv_1^s\bfv_3^m\bfv_4^n\bfv_5^p\big(K^k
		(\alpha_{s,m,n,p,k}
		+c\beta_{s,m,n,p,k})
		\blactright E\big)\\
		&=\textstyle\sum \bfv_1^s\bfv_3^m\bfv_4^n\bfv_5^p\bfv_3K^{k+1}\big((q^{2k}-1)\alpha_{s,m,n,p,k}
		+c(q^{2k}-q^{-1})\beta_{s,m,n,p,k}\big)
	\end{split}
\end{equation*}
Thus if $\alpha_{s,m,n,p,k}\neq 0$ then $q^{2k}=1$, so $k=0$ since $q$ is not a root of unity. Also if $\beta_{s,m,n,p,k}\neq 0$ then $2k=-1$ which is impossible for $k\in\ZZ$.
\end{proof}

Now that we have a generating set for $^+\cA[\bfv_3^{-1},c^{-1}]^+$, we will step back to $^+\cA[\bfv_3^{-1}]^+$ by determining which vectors in $^+\cA[\bfv_3^{-1},c^{-1}]^+$ can be expressed without $c^{-1}$.

\begin{proposition}\label{thm:hw_localized_v3}
The algebra $^+\cA[\bfv_3^{-1}]^+$ is generated by $\{ \bfv_1,\bfv_3^{\pm 1},\bfv_4,\bfv_5,\bfv_6 \}$.
\end{proposition}

\begin{proof}
We wish to find $\big\langle \bfv_1,\bfv_3^{\pm 1},\bfv_4,\bfv_5^{\pm 1} \big\rangle\cap A[\bfv_3^{-1}]$.  We notice that $c$ does not occur in $\bfv_3$ or $\bfv_4$, but only in $\bfv_1$ and $\bfv_5$.  Therefore, if $\bfv_1^s\bfv_3^m\bfv_4^n\bfv_5^{-1} \in A[\bfv_3^{-1}]$ then $s>0$.  In fact, it must be that $s\geq 2$ so that $c^2$ is a factor in $\bfv_1^s\bfv_3^m\bfv_4^n$.  Since $s\geq 2$, we may substitute $\bfv_1^2\bfv_5^{-1}=\bfv_6$.
\end{proof}

Our final task is to step back to $\HCpp$ by determining which vectors in $^+\cA[\bfv_3^{-1}]^+$ can be expressed without $\bfv_3^{-1}$.  This will be much more challenging than removing $c^{-1}$ (Proposition \ref{thm:hw_localized_v3}), because $\bfv_3^{-1}$ is what gave the algebra quasi-commuting generators.

\begin{definition}
Recall that $\{E^mF^nK^{\pm p}a^\ell c^k\}$ is a basis for $H\otimes \Cplus$ (See Lemma~\ref{lem:cplus}).
Define $\lambda$ to be the projection of $H\otimes\Cplus$ onto the subalgebra $\lr{F^nK^{\pm p}a^\ell c^k}$ given by 
\begin{equation*}
	\lambda\big(E^mF^nK^{\pm p}a^\ell c^k\big)=
	\begin{cases}
		F^nK^{\pm p}a^\ell c^k & \text{if $m=0$,}\\
		0 & \text{if $m>0$.}
	\end{cases}
\end{equation*}
\end{definition}

\begin{remark}
One may view $\lambda$ as the quotient by the right ideal generated by $E$.
\end{remark}

\begin{remark}
If $B$ is the algebra $\lr{F,K^{\pm 1},a,c}$, 
then $\lambda\colon H\otimes\Cplus\to B$ is a morphism of right $B$-modules.
\end{remark}

In particular,
\begin{subequations}\label{residues}
\begin{align}
	\lambda(\bfv_1) &= -c^2\\
	\lambda(\bfv_3) &= 0\\
	\lambda(\bfv_4) &= \tfrac{q^{-1}}{(q-q^{-1})^2}K+\tfrac{q}{(q-q^{-1})^2}K^{-1}\\
	\lambda(\bfv_5) &= K^{-1}c^2\\
	\lambda(\bfv_6) &= Kc^2
\end{align}
\end{subequations}
Now $\lambda(\bfv_2^2)\neq \lambda(\bfv_2)\lambda(\bfv_2)$ because of the complicated relations between $E$ and $F$, so $\lambda$ is not a morphism of algebras. But if we avoid $\bfv_2$ then we see some useful structure:

\begin{proposition}\label{prop:v_homomorphism}
Let $R$ be the algebra generated by $\{ \bfv_1,\bfv_3,\bfv_4,\bfv_5,\bfv_6 \}$.  Let $M$ be the left $R$-module with $R$-basis $\{\bfv_2^n\}$ where the action is $r\actleft r'\bfv_2^n=(rr')\bfv_2^n$.  Then the restriction of $\lambda$ to $M$ is an $R$-module homomorphism, where the action of $R$ on $\lambda(M)$ is $r\actleft v=\lambda(r)v$.
\end{proposition}

\begin{proof}
Because $K^{\pm 1}$, $a$, and $c$ each quasi-commute with $E$, we conclude that $\lambda(\bfv_i\bfv_j)=\lambda(\bfv_i)\lambda(\bfv_j)$ for $i,j\in\{1,3,4,5,6\}$ since $F$ does not appear in any of those vectors.  Thus $\lambda$, when restricted to $R$, is a ring homomorphism.  By the same argument, $\lambda(\bfv_i\bfv_2^n)=\lambda(\bfv_i)\lambda(\bfv_2^n)$ when $i\neq 2$.
\end{proof}

\begin{lemma}\label{lem:divisibility}
Let $u\in \lr{ \bfv_1,\bfv_2,\bfv_4,\bfv_5,\bfv_6 }\subset H\otimes\Cplus$.  Then $u$ can be written as a finite sum $u=\sum \bfv_3^iu_i$ where each $u_i$ belongs to the span of
\[
	S=\{\bfv_1^m\bfv_5^p\bfv_2^s\}
	\cup
	\{\bfv_1^m\bfv_6^r\bfv_2^s\}
	\cup
	\{\bfv_4^n\bfv_5^p\bfv_2^s\}
	\cup
	\{\bfv_4^n\bfv_6^p\bfv_2^s\}.
\]
\end{lemma}

\begin{proof}
Recall that $\lr{\bfv_1,\bfv_4,\bfv_5,\bfv_6}$ is commutative, and recall that $\bfv_2$ and $\bfv_3$ each quasi-commute with $\bfv_1$, $\bfv_4$, $\bfv_5$, and $\bfv_6$, but not with each other. 
We simply replace all occurrences of $\bfv_5\bfv_6$ and $\bfv_1\bfv_4$ using the identities
\[
	\bfv_5\bfv_6=\bfv_1^2\qquad\text{and}\qquad 
	\bfv_1\bfv_4=\bfv_3\bfv_2-\tfrac{q}{(q-q^{-1})^2}\bfv_5-\tfrac{q^{-1}}{(q-q^{-1})^2}\bfv_6.
\]
This introduces $\bfv_3$'s, but only such that they appear to the left of the $\bfv_2$'s.  For example,
\begin{equation*}
	\begin{split}
		\bfv_1^2\bfv_4^2= \bfv_1(\bfv_1\bfv_4)\bfv_4
		&=\bfv_1\big(\bfv_3\bfv_2-\tfrac{q}{(q-q^{-1})^2}\bfv_5-\tfrac{q^{-1}}{(q-q^{-1})^2}\bfv_6\big)\bfv_4\\
		&=\bfv_3(\bfv_1\bfv_4)\bfv_2-\tfrac{q}{(q-q^{-1})^2}(\bfv_1\bfv_4)\bfv_5-\tfrac{q^{-1}}{(q-q^{-1})^2}(\bfv_1\bfv_4)\bfv_6
	\end{split}
\end{equation*} and so on. One must be careful not to group this as $\bfv_1^2\bfv_4^2=(\bfv_1\bfv_4)(\bfv_1\bfv_4)$ and replace both at once.
\end{proof}

\begin{proposition}\label{prop:v_injective}
The set $S$ from Lemma \ref{lem:divisibility} is linearly independent, and the restriction of $\lambda$ to the span of $S$ is injective.
\end{proposition}

Before we prove this proposition, let us consider how it will be used to prove Theorem \ref{thm:v1v2v3v4generate_hw}. By Proposition \ref{thm:hw_localized_v3} and the quasi-commutativity of the vectors $\bfv_1$, $\bfv_3$, $\bfv_4$, $\bfv_5$, and $\bfv_6$, we may write any highest-weight bivector $w\in{^+}A[\bfv_3^{-1}]^+\cap (H\otimes C)$ as
\[
	w=u_0+\sum_{j=1}^m\bfv_3^{-i}u_j,\quad\text{$u_0\in\lr{ \bfv_1,\bfv_3,\bfv_4,\bfv_5,\bfv_6 }$,\quad$u_1,\dotsc,u_m\in\lr{ \bfv_1,\bfv_4,\bfv_5,\bfv_6 }$.}
\]
Now we apply Lemma \ref{lem:divisibility} to $u_1,\dotsc,u_m$ and cancel $\bfv_3^{-1}\bfv_3$ in the above sum wherever possible.  Regrouping terms, and recalling $\bfv_5,\bfv_6\in\lr{\bfv_1,\bfv_2,\bfv_3,\bfv_4}$, we get a new sum
\[
	w=y_0+\sum_{i=1}^n\bfv_3^{-i}y_i,\quad\text{$y_0\in\lr{ \bfv_1,\bfv_2,\bfv_3,\bfv_4 }$,\quad$y_1,\dotsc,y_n\in\Span(S)$.}
\]
Now $\bfv_3^{n-1}w=z+\bfv_3^{-1}y_n$ where $z\in\lr{ \bfv_1,\bfv_2,\bfv_3,\bfv_4 }$, so $y_n=\bfv_3^nw-\bfv_3z$.  Then $\lambda(y_n)=\lambda(\bfv_3^nw-\bfv_3z)=0$.  Since $y_n\in\Span(S)$, it follows from Proposition \ref{prop:v_injective} that $y_n=0$.  Similarly $y_{n-1}=y_{n-2}=\dotsb=y_1=0$.  Thus $w=y_0$, so $w\in\lr{ \bfv_1,\bfv_2,\bfv_3,\bfv_4 }$.

Not only does this show that $\HCpp=\lr{ \bfv_1,\bfv_2,\bfv_3,\bfv_4 }$, but Proposition \ref{prop:v_injective} also implies that
\begin{align*}
	\{\bfv_3^\ell \bfv_1^m \bfv_5^p \bfv_2^s \}
	\cup\{\bfv_3^\ell \bfv_1^m \bfv_6^r \bfv_2^s \}
	\cup\{\bfv_3^\ell \bfv_4^n \bfv_5^p \bfv_2^s \}
	\cup\{\bfv_3^\ell \bfv_4^n \bfv_6^r \bfv_2^s \}
\end{align*}
is linearly independent, proving Theorem \ref{thm:v1v2v3v4generate_hw}.  It remains to prove Proposition \ref{prop:v_injective}.

\begin{proof}
Let $u$ belong to the span of $S$, and suppose that $\lambda(u)=0$.  We may write $u=\sum_{i\in\Omega} \alpha_iu_i\bfv_2^{s_i}$ where each coefficient $\alpha_i$ is nonzero, each $u_i$ is a vector belonging to the set
\[
	S'=\{\bfv_1^m\bfv_5^p\}
	\cup
	\{\bfv_1^m\bfv_6^r\}
	\cup
	\{\bfv_4^n\bfv_5^p\mid n>0\}
	\cup
	\{\bfv_4^n\bfv_6^r\mid n>0\},
\]
and $u_i\bfv_2^{s_i}= u_j\bfv_2^{s_j}$ only if $i=j$. We will show that the index set $\Omega$ is empty, which will prove both parts of the proposition at once.

Suppose $\Omega\neq\emptyset$.  Let $s=\max\{s_i\}_{i\in\Omega}$ and $\Psi=\{i\in\Omega\mid s_i=s\}$.  (We will actually show that $\Psi$ is empty, which is a contradiction, implying $\Omega=\emptyset$.)  Recall that 
\begin{align*}
	\bfv_2&=(q-q^{-1})\Delta ac-\tfrac{q+q^{-1}}{q-q^{-1}}Kac-q^{-2}FKc^2+Ea^2.
\end{align*}
Recall also that $\{E^i F^j \Delta^k K^{\pm\ell} a^p c^r b^m d^n\mid ij=pn=0\}$ is a basis for $H\otimes C$, and let $\pi_s$ be the projection onto the subspace where $j=s$.  Because $F$ does not appear in any $u_i$ and because $s_i<s$ for all $i\in\Omega\setminus\Psi$, we have 
\begin{equation*}
	\pi_s(u)
	=\pi_s\big(\textstyle\sum_{i\in\Omega} \alpha_iu_i\bfv_2^{s_i}\big)
	=\pi_s\big(\textstyle\sum_{i\in\Psi}\alpha_i u_i \bfv_2^s\big)
	=\textstyle\sum_{i\in\Psi}\alpha_i u_i (-q^{-2}FKc^2)^s.
\end{equation*}
By Proposition \ref{prop:v_homomorphism}, 
\begin{align*}
		\pi_s\big(\lambda(u)\big)
		&=\pi_s\big(\textstyle\sum_{i\in\Omega}\alpha_i\lambda(u_i)\lambda(\bfv_2^{s_i})\big)
		=\textstyle\sum_{i\in\Psi}\alpha_i\lambda(u_i)(-q^{-2}FKc^2)^s.
\end{align*}
Because $\lambda(u)=0$, it follows that
\begin{align}\label{eqn:simplified_linear_combination}\textstyle\sum_{i\in\Psi}\alpha_i\lambda(u_i)=0.\end{align}
We will show that the vectors $\lambda(u_i)$ are linearly independent, so $\Psi=\emptyset$.

Because $\{u_i\}\subset S'$, we apply \eqref{residues} to get
\begin{subequations}
\begin{align}
	\lambda(\bfv_1^m\bfv_5^p) &= (-1)^mK^{-p}c^{2(m+p)}\label{proj_v1v5}\\
	\lambda(\bfv_1^m\bfv_6^r) &= (-1)^mK^{r}c^{2(m+r)}\label{proj_v1v6}\\
	\lambda(\bfv_4^n\bfv_5^p) &=
		\sum_{j=0}^n\binom{n}{j}\frac{q^{n-2j}}{(q-q^{-1})^{2n}}K^{-(n-2j)-p}c^{2p}\label{proj_v4v5}\\
	\lambda(\bfv_4^n\bfv_6^r) &=
		\sum_{j=0}^n\binom{n}{j}\frac{q^{n-2j}}{(q-q^{-1})^{2n}}K^{-(n-2j)+r}c^{2r}\label{proj_v4v6}
\end{align}
\end{subequations}
Define \(N=\max\big\{n\;\big\vert\;\text{$\bfv_4^n\bfv_5^p\in\{u_i\}_{i\in\Psi}$ for some $p$ or $\bfv_4^n\bfv_6^r\in\{u_i\}_{i\in\Psi}$ for some $r$}\big\}\) and suppose that $N>0$.

If $u_i=\bfv_4^N\bfv_5^p$ for some $i\in\Psi$, then \eqref{proj_v4v5} shows that one of the summands ($j=0$) of $\lambda(u_i)$ is $q^N(q-q^{-1})^{-2n}K^{-(N+p)}c^{2p}$.  Now \eqref{eqn:simplified_linear_combination} shows that $\lambda(u_i)$ has linear dependence with the other terms, so some $\lambda(u_{i'})$ must include a nonzero multiple of $K^{-(N+p)}c^{2p}$.  We note that $2p<2(N+p)$, and neither \eqref{proj_v1v5} nor \eqref{proj_v1v6} can produce a term whose powers of $c$ and $K$ have this property, and \eqref{proj_v4v5} cannot either because $n\leq N$.  Now \eqref{proj_v4v6} shows that $\lambda(\bfv_4^n\bfv_6^r)$ can produce such a term, but only if $r=p$ and $n=N+2j+2p$ for some $j\geq 0$.  Because $N$ was chosen to be maximal, we would need $j=p=0$. But then $r=p=0$ and $n=N$, so $\bfv_4^N\bfv_6^0=\bfv_4^N\bfv_5^0$ is just the original $u_i$.  Therefore $\lambda(u_i)$ is linearly independent of the other possible summands, so \eqref{eqn:simplified_linear_combination} implies we cannot have $u_i=\bfv_4^N\bfv_5^p$ for any $i\in\Psi$.

Similarly, if $u_i=\bfv_4^N\bfv_6^r$ for some $i\in\Psi$, then \eqref{proj_v4v6} shows that one of the summands ($j=N$) of $\lambda(u_i)$ is $q^{-n}(q-q^{-1})^{-2n}K^{n+r}c^{2r}$.  Such a term could be produced only by \eqref{proj_v4v5} when $p=r$ and $2j-n=N+2r$ for some $j\leq n$, but because $N$ is maximal we would need $j=n=N$ and $r=p=0$, in which case $\bfv_4^N\bfv_5^0=\bfv_4^N\bfv_6^0$ is just the original $u_i$.  Therefore for \eqref{eqn:simplified_linear_combination} to hold, we cannot have $u_i=\bfv_4^N\bfv_6^r$ for any $i\in\Psi$.

We have shown that $N$ must be zero, so $\{u_i\}_{i\in\Psi}\subset\{\bfv_1^m\bfv_5^p\}\cup\{\bfv_1^{m'}\bfv_6^r\}$. But the powers of $K$ and $c$ in \eqref{proj_v1v5} and \eqref{proj_v1v6} show that $\lambda(\bfv_1^m\bfv_5^p)$ and $\lambda(\bfv_1^{m'}\bfv_6^r)$ are linearly dependent only when $p=r=0$ and $m=m'$.  Therefore, \eqref{eqn:simplified_linear_combination} implies $\Psi=\emptyset$.  We conclude that $\Omega=\emptyset$.
\end{proof}

\subsection{Proof of Corollary \ref{cor:hilbert_series}} We count the vectors of degree $n$ in the basis given in Theorem \ref{thm:v1v2v3v4generate_hw} where the vectors $\bfv_5$ and $\bfv_6$ have degree 2.  There are $\binom{n+2}{2}$ vectors of degree $n$ in $\{\bfv_3^\ell \bfv_1^m \bfv_2^s\}$, and $\binom{n+2}{2}$ vectors of degree $n$ in $\{\bfv_3^\ell \bfv_4^m \bfv_2^s\}$, and $\binom{n+1}{1}$ vectors of degree $n$ in the intersection $\{\bfv_3^\ell \bfv_2^s\}$.  This gives
\[ \tbinom{n+2}{2}+\tbinom{n+2}{2}-\tbinom{n+1}{1}=(n+1)^2 \]
vectors of degree $n$ in $\{\bfv_3^\ell \bfv_1^m \bfv_2^s\}\cup\{\bfv_3^\ell \bfv_4^m \bfv_2^s\}$.

Now we consider $\{\bfv_3^\ell \bfv_1^m \bfv_5^p \bfv_2^s\}\cup\{\bfv_3^\ell \bfv_4^m \bfv_5^p \bfv_2^s\}$.  There are $((n-2)+1)^2=(n-1)^2$ vectors of degree $n$ with $p=1$ because $\deg(\bfv_5)=2$, and there are $(n-3)^2$ vectors of degree $n$ with $p=2$, and so on.

The same happens when we introduce $\bfv_6$.  Therefore the number of vectors of degree $n$ in the full basis is
\begin{align}\label{eqn:counting_hw_bivectors}
	(n+1)^2+2(n-1)^2+2(n-3)^2+2(n-5)^2+\dotsb+
	\begin{cases}
		2(1)^2&\text{if $n$ is even,}\\
		2(2)^2&\text{if $n$ is odd.}
	\end{cases}
\end{align}
For even $n$, this is
\begin{equation*}
	\begin{split}
	(n+1)^2+2\sum_{i=1}^{\frac{n}{2}}(2i-1)^2
	&= (n+1)^2
		+8\sum_{i=1}^{\frac{n}{2}}i^2
		-8\sum_{i=1}^{\frac{n}{2}}i
		+2\sum_{i=1}^{\frac{n}{2}}1\\
	&= (n+1)^2
		+8\bigg(\frac{n^3}{24}+\frac{n^2}{8}+\frac{n}{12}\bigg)
		-8\bigg(\frac{n^2}{8}+\frac{n}{4}\bigg)
		+2\bigg(\frac{n}{2}\bigg)\\
	&= \tfrac{1}{3}(n^2+2n+3)(n+1).
	\end{split}
\end{equation*}
For odd $n$, the sum $(n+1)^2+2\sum_{i=1}^{(n-1)/2}(2i)^2$ is also $\tfrac{1}{3}(n^2+2n+3)(n+1)$.  Therefore, the Hilbert series is
\begin{equation*}
	\begin{split}
		h(t)
		&= \frac{1}{3}\sum_{n=0}^\infty (n^2+2n+3)(n+1)t^n\\
		&= \frac{1}{3}\sum_{n=0}^\infty \big[(n+3)(n+2)(n+1)-3(n+2)(n+1)+3(n+1)\big]t^n\\
		&= \frac{1}{3}\cdot\frac{d^3}{dt^3}\bigg(\sum_{n=0}^\infty t^n\bigg)
			-\frac{d^2}{dt^2}\bigg(\sum_{n=0}^\infty t^n\bigg)
			+\frac{d}{dt}\bigg(\sum_{n=0}^\infty t^n\bigg)\\
		&= \frac{1}{3}\cdot\frac{d^3}{dt^3}\bigg(\frac{1}{1-t}\bigg)
			-\frac{d^2}{dt^2}\bigg(\frac{1}{1-t}\bigg)
			+\frac{d}{dt}\bigg(\frac{1}{1-t}\bigg)\\
		&= \frac{2}{(1-t)^4}-\frac{2}{(1-t)^3}+\frac{1}{(1-t)^2}.
	\end{split}
\end{equation*}

\subsection{Proof of Main Theorem \ref{thm:main_theorem}}

For each pair of nonnegative integers $\lambda$ and $\mu$ satisfying $\lambda-\mu\in 2\ZZ$, we define $\cH_{\lambda,\mu}$ to be the $\rD(\rU_q(\fsl_2))$-bimodule generated as follows:
\begin{align*}
\cH_{\lambda,\mu}=
	\begin{cases}
		\big\langle\bfv_1^{\mu}\bfv_5^{(\lambda-\mu)/2}\big\rangle&\text{if $\lambda\geq \mu$,}\\
		\big\langle\bfv_1^{\lambda}\bfv_6^{(\mu-\lambda)/2}\big\rangle&\text{if $\lambda<\mu$.}
	\end{cases}
\end{align*}

By Lemma \ref{lem:hw_generate_locally_finite_modules} and Corollary \ref{cor:Hpp_equals_HCpp}, each vector in $\cH$ belongs to a sub-bimodule generated by vectors in $\HCpp$. By Theorem \ref{thm:v1v2v3v4generate_hw}, $\HCpp\subset\sum_{n=0}^\infty (\cH_{1,1})^n$.

At this point it should not be surprising that $\sum_{n=0}^\infty (\cH_{1,1})^n$ can be written as an internal direct sum $\bigoplus \cH_{\lambda,\mu}$ for the following reason. We believe that $\cH_{\lambda,\mu}=\beta_{V_{\lambda,\mu}}(V_{\lambda,\mu}\otimes V_{\lambda,\mu}^*)$, and we know that products $V_{\lambda_1,\mu_1}\otimes V_{\lambda_2,\mu_2}$ obey the Pierri Rule, implying that tensor powers of $V_{1,1}$ span the subalgebra $\bigoplus V_{\lambda,\mu}$ where $\lambda-\mu$ is even.  Therefore powers of $\beta_{V_{1,1}}(V_{1,1}\otimes V_{1,1}^*)$ decompose as sums of $\beta_{V_{\lambda,\mu}}(V_{\lambda,\mu}\otimes V_{\lambda,\mu}^*)$ where $\lambda-\mu\in 2\ZZ$.

We will verify that $\bfv_1^{\mu}\bfv_5^{s}$ and $\bfv_1^{\lambda}\bfv_6^{s}$ generate simple left $\rD(\rU_q(\fsl_2))$-modules with $\mu+1$ and $\lambda+1$ highest-weight vectors, respectively, and \eqref{eqn:counting_hw_bivectors} shows that this accounts for all highest-weight bivectors of appropriate degree. 

The relations $Ec=qcE$ and $Kc=q^2cK$ (See Example \ref{ex:sl2_double_relations}) prove the following lemma.

\begin{lemma}\label{lem:c_action_preserves_hw}
Let $v$ be a highest-weight vector of weight $m$ in a left $\rD(\rU_q(\fsl_2))$-module.  If $c\blactleft v$ is nonzero, then it is a highest-weight vector of weight $m+2$.  (A similar result holds for right modules.)
\end{lemma} 

First we consider the case $\lambda=\mu+2s\geq\mu$.  We will show that $\big\langle\bfv_1^{\mu}\bfv_5^{(\lambda-\mu)/2}\big\rangle$ has highest-weight vectors $\{\bfv_3^{\mu-n}\bfv_1^{n}\bfv_5^s\mid n=0,1,\dotsc,\mu\}$ and that this module is simple. 
The following lemmas are proved by straightforward but tedious calculations from the actions of $\rD(\rU_q(\fsl_2))$ on various products of vectors.  Those actions are included at the end of the section.

\begin{lemma}\label{lem:rect1byc}
For any collection of nonnegative integers $\ell$, $n$, and $s$,
\begin{align*}
	c\blactleft (\bfv_3^\ell \bfv_1^n \bfv_5^s) 
	&= q^s[\ell]_q \bfv_3^{\ell-1}\bfv_1^{n+1}\bfv_5^s.
\end{align*}
\end{lemma}

\begin{lemma}\label{lem:rect1leftbyFbad}
For any collection of nonnegative integers $\ell$, $m$, $n$, and $s$ we have the following four left actions:
\begin{align*}
F \blactleft (\bfv_3^\ell \bfv_1^n \bfv_5^s) &= 
	(1-q^{-2n-4s})\bfv_3^{\ell+1}\bfv_1^{n-1}\bfv_5^s
	+q^{1-2n-2s}[2n+2s]_q
	\bfv_3^\ell v_{41} \bfv_1^{n-1} \bfv_5^s\\
b \blactleft (\bfv_3^\ell \bfv_1^n \bfv_5^s) &= 
	(q^{-2n-\ell}-q^{\ell})q^{1-s}\bfv_3^\ell v_{41}\bfv_1^{n-1}\bfv_5^s
	-q^{2-2n-s}[\ell]_q \bfv_3^{\ell-1}v_{41}^2\bfv_1^{n-1}\bfv_5^s\\
a\blactleft (\bfv_3^\ell v_{41}^m \bfv_1^n \bfv_5^s) &= 
	q^{m-\ell+s}\bfv_3^\ell v_{41}^m\bfv_1^n\bfv_5^s
	-q^{m+s+1}[\ell]_q \bfv_3^{\ell-1}v_{41}^{m+1}\bfv_1^n\bfv_5^s\\
d\blactleft (\bfv_3^\ell v_{41}^m \bfv_1^n \bfv_5^s) &= 
	q^{\ell-m-s}\bfv_3^\ell v_{41}^m \bfv_1^n\bfv_5^s
	+q^{-m-2n-s-1}[\ell]_q \bfv_3^{\ell-1}v_{41}^{m+1}\bfv_1^n\bfv_5^s.
\end{align*}
\end{lemma}\noindent
The proof uses the relation
\(
\bfv_1 \vd_{21} = (q-q^{-1})q^{-2}\bfv_3\bfv_5+q^{-3}v_{41}\bfv_5
\) 
to calculate the action of $F$ (The bivectors $v_{41}$ and $\vd_{21}$ are from Examples \ref{ex:h11} and \ref{ex:h20}).  This lemma implies that
\begin{multline}\label{eqn:special_left_action_1}
\Big(q^{2+s}\big(q^{-\ell+s-1}d-q^{\ell-s+1}a\big)F
	+[2n+2s]_q\big(q^{\ell+4}+q^{-\ell-2n}\big)b\Big)
\blactleft \bfv_3^\ell \bfv_1^n \bfv_5^s\\
	= q^{1-2\ell-5n-3s}(1+q^{2\ell+2n+2}-q^{2\ell+4n+4s+2}-q^{4\ell+6n+4s+4})[n]_q
	\bfv_3^\ell v_{41} \bfv_1^{n-1} \bfv_5^s.
\end{multline}
Because $q$ is generic and $\ell,n,s\geq 0$, \eqref{eqn:special_left_action_1} is zero if and only if $n=0$.  Therefore the vector $\bfv_3^{\ell+1}\bfv_1^{n-1}\bfv_5^s$ is in the span of \eqref{eqn:special_left_action_1} and $F \blactleft (\bfv_3^\ell \bfv_1^n \bfv_5^s)$, proving the following.

\begin{corollary}\label{cor:down_weight_1}
If a left submodule of $\cH_{\mu+2s,\mu}$ contains the vector $\bfv_3^\ell \bfv_1^n \bfv_5^s$ where $n\geq 1$, then this submodule also contains the vector $\bfv_3^{\ell+1}\bfv_1^{n-1}\bfv_5^s$.
\end{corollary}

This shows that $\{\bfv_3^{\mu-n}\bfv_1^{n}\bfv_5^s\mid n=0,1,\dotsc,\mu\}$ is the set of highest-weight vectors in $\big\langle\bfv_1^{\mu}\bfv_5^{(\lambda-\mu)/2}\big\rangle$ and that this module is simple. 
We use similar reasoning for the case $\mu=\lambda+2s\geq\lambda$, but $\bfv_6$ takes the place of $\bfv_5$.

\begin{lemma}\label{lem:rect2byc}
For any collection of nonnegative integers $\ell$, $n$, and $s$,
\begin{align*}
	c\blactleft (\bfv_3^\ell \bfv_1^n \bfv_6^s) 
	&= q^{-s}[\ell]_q \bfv_3^{\ell-1}\bfv_1^{n+1}\bfv_6^s.
\end{align*}
\end{lemma}

\begin{lemma}\label{lem:rect2leftbyFbad}
For any collection of nonnegative integers $\ell$, $m$, $n$, and $s$ we have the following four left actions:
\begin{align*}
F\blactleft(\bfv_3^\ell \bfv_1^n \bfv_6^s) &= 
	(1-q^{-2n})\bfv_3^{\ell+1}\bfv_1^{n-1}\bfv_6^s
	+q^{1-2n-2s}[2n+2s]_q
	\bfv_3^\ell v_{41}\bfv_1^{n-1}\bfv_6^s\\
b\blactleft(\bfv_3^\ell \bfv_1^n \bfv_6^s) &= (q^{-\ell-2n-3s}-q^{\ell+s})q\bfv_3^{\ell}v_{41}\bfv_1^{n-1}\bfv_6^s-q^{2-2n-3s}[\ell]_q
\bfv_3^{\ell-1}v_{41}^2\bfv_1^{n-1}\bfv_6^s\\
a\blactleft(\bfv_3^\ell v_{41}^m \bfv_1^n \bfv_6^s) &= q^{m-\ell-s}\bfv_3^\ell v_{41}^m \bfv_1^n \bfv_6^s - q^{m-s+1}[\ell]_q
\bfv_3^{\ell-1}v_{41}^{m+1}\bfv_1^n \bfv_6^s\\
d\blactleft(\bfv_3^\ell v_{41}^m \bfv_1^n \bfv_6^s) &= q^{\ell-m+s}\bfv_3^\ell v_{41}^m \bfv_1^n \bfv_6^s
+q^{-m-2n-3s-1}[\ell]_q
\bfv_3^{\ell-1} v_{41}^{m+1} \bfv_1^n \bfv_6^s.
\end{align*}
\end{lemma}\noindent
The proof uses the relation
\(
\bfv_1\vdd_{21} = q^{-3}v_{41}\bfv_6
\) 
to calculate the action of $F$ (The bivector $\vdd_{21}$ is from Example \ref{ex:h02}). This lemma implies that
\begin{multline}\label{eqn:special_left_action_2}
\Big(q^{2-s}\big(q^{-\ell-s-1}d-q^{\ell+s+1}a\big)F
	+[2n+2s]_q\big(q^{\ell+4}+q^{-\ell-2n-4s}\big)b\Big)
	\blactleft \bfv_3^\ell \bfv_1^n \bfv_6^s\\
= q^{1-2\ell-5n-7s}\big(1+(1-q^{2n})q^{2\ell+2n+4s}-q^{4\ell+6n+8s+4}\big)[n+2s]_q\bfv_3^{\ell}v_{41}\bfv_1^{n-1}\bfv_6^s.
\end{multline}
Because $q$ is generic and $\ell,n,s\geq 0$, \eqref{eqn:special_left_action_2} is zero if and only if $n=s=0$.  Therefore the vector $\bfv_3^{\ell+1}\bfv_1^{n-1}\bfv_6^s$ is in the span of \eqref{eqn:special_left_action_2} and $F \blactleft (\bfv_3^\ell \bfv_1^n \bfv_6^s)$, proving the following.

\begin{corollary}\label{cor:down_weight_2}
If a left submodule of $\cH_{\lambda,\lambda+2s}$ contains the vector $\bfv_3^\ell \bfv_1^n \bfv_6^s$ where $n\geq 1$, then this submodule also contains the vector $\bfv_3^{\ell+1}\bfv_1^{n-1}\bfv_6^s$.
\end{corollary}

This shows that $\{\bfv_3^{\lambda-n}\bfv_1^n\bfv_6^s\mid n=0,1,\dotsc,\lambda\}$ is the set of highest-weight vectors in $\big\langle\bfv_1^{\lambda}\bfv_6^{(\mu-\lambda)/2}\big\rangle$ and that this module is simple.

We conclude by listing the actions needed to verify Lemmas \ref{lem:rect1byc}, \ref{lem:rect1leftbyFbad}, \ref{lem:rect2byc}, and \ref{lem:rect2leftbyFbad}. The following lemmas are all proved by induction using the actions given in Examples \ref{ex:h11}, \ref{ex:h20}, and \ref{ex:h02}.

\begin{lemma}\label{lem:v11leftca}
For any nonnegative integer $n$,
\begin{align*}
a\blactleft \bfv_3^n &= q^{-n}\bfv_3^n - q[n]_q\bfv_3^{n-1}v_{41}
&c\blactleft \bfv_3^n &= [n]_q\bfv_3^{n-1}\bfv_1\\
b\blactleft \bfv_3^n &= [n]_q\bfv_3^{n-1}v_{31}
&d\blactleft \bfv_3^n &= q^n\bfv_3^n+q^{-1}[n]_q\bfv_3^{n-1}v_{41}.
\end{align*}
\end{lemma}\noindent
The proof uses the relations $\bfv_1\bfv_3=\bfv_3\bfv_1$, $v_{31}\bfv_3=\bfv_3v_{31}$, $v_{41}\bfv_3=\bfv_3v_{41}$, $v_{31}\bfv_1=-(q-q^{-1})q\bfv_3v_{41}-q^2v_{41}^2$, $\bfv_1v_{31}=-(q-q^{-1})q^{-1}\bfv_3v_{41}-q^{-2}v_{41}^2$, $v_{41}\bfv_1=q^2\bfv_1v_{41}$, and $v_{31}v_{41}=q^2v_{41}v_{31}$.

\begin{lemma}\label{lem:v21leftabcd}
For any nonnegative integer $n$,
\begin{align*}
a\blactleft \bfv_1^n &= \bfv_1^n
&b\blactleft \bfv_1^n &= (q^{-2n}-1)qv_{41}\bfv_1^{n-1}
&c\blactleft \bfv_1^n &= 0
&d\blactleft \bfv_1^n &= \bfv_1^n.
\end{align*}
\end{lemma}\noindent
The proof uses $v_{41}\bfv_1=q^2\bfv_1v_{41}$.

\begin{lemma}\label{lem:v21leftf}
For any nonnegative integer $n$,
\begin{equation*}
	F\blactleft \bfv_1^n = (1-q^{-2n})\bfv_3\bfv_1^{n-1}+q^{1-2n}[2n]_qv_{41}\bfv_1^{n-1}.
\end{equation*}
\end{lemma}\noindent
The proof uses $\bfv_1\bfv_3=\bfv_3\bfv_1$ and $v_{41}\bfv_1=q^2\bfv_1v_{41}$.

\begin{lemma}\label{lem:v41leftabcd}
For any nonnegative integer $n$,
\begin{align*}
a\blactleft v_{41}^n &= q^nv_{41}^n
&b\blactleft v_{41}^n &= 0
&c\blactleft v_{41}^n &= 0
&d\blactleft v_{41}^n &= q^{-n}v_{41}^n.
\end{align*}
\end{lemma}

\begin{lemma}\label{lem:vd11leftabcd}
For any nonnegative integer $n$,
\begin{align*}
a\blactleft \bfv_5^n &= q^n \bfv_5^n
&b\blactleft \bfv_5^n &= 0
&c\blactleft \bfv_5^n &= 0
&d\blactleft \bfv_5^n &= q^{-n} \bfv_5^n.
\end{align*}
\end{lemma}

\begin{lemma}\label{lem:vd11leftbyf}
For any nonnegative integer $n$, $F\blactleft \bfv_5^n = q^{2-2n}[2n]_q \vd_{21} \bfv_5^{n-1}$.
\end{lemma}\noindent
The proof uses $\vd_{21}\bfv_5=q^2\bfv_5\vd_{21}$.

\begin{lemma}\label{lem:vdd11leftabcd}
For any nonnegative integer $n$,
\begin{align*}
a\blactleft \bfv_6^n &= q^{-n}\bfv_6^n,
&b\blactleft \bfv_6^n &= q^{n+2}(q^{-4n}-1)\vdd_{21}\bfv_6^{n-1},
&c\blactleft \bfv_6^n &= 0,
&d\blactleft \bfv_6^n &= q^n \bfv_6^n.
\end{align*}
\end{lemma}\noindent
The proof uses $\vdd_{21}\bfv_6=q^2\bfv_6\vdd_{21}$.

\begin{lemma}\label{lem:vdd11leftbyf}
For any nonnegative integer $n$, $F \blactleft \bfv_6^n = q^{2-2n}[2n]_q\vdd_{21}\bfv_6^{n-1}$.
\end{lemma}\noindent
The proof uses $\vdd_{21}\bfv_6=q^2\bfv_6\vdd_{21}$.

\bibliography{Cite}

\begin{bibdiv}
\begin{biblist}

\bib{MR2852300}{article}{
      author={Berenstein, Arkady},
      author={Greenstein, Jacob},
       title={Quantum folding},
        date={2011},
        ISSN={1073-7928},
     journal={Int. Math. Res. Not. IMRN},
      number={21},
       pages={4821\ndash 4883},
         url={http://dx.doi.org/10.1093/imrn/rnq264},
      review={\MR{2852300 (2012m:17019)}},
}

\bib{MR0323842}{book}{
      author={Humphreys, James~E.},
       title={Introduction to {L}ie algebras and representation theory},
   publisher={Springer-Verlag},
     address={New York},
        date={1972},
        note={Graduate Texts in Mathematics, Vol. 9},
      review={\MR{0323842 (48 \#2197)}},
}

\bib{jordan2011}{misc}{
      author={Jordan, David},
       title={The peter-weyl theorem for classical and quantum sl(n)},
        date={2011},
}

\bib{MR1381692}{book}{
      author={Majid, Shahn},
       title={Foundations of quantum group theory},
   publisher={Cambridge University Press},
     address={Cambridge},
        date={1995},
        ISBN={0-521-46032-8},
         url={http://dx.doi.org/10.1017/CBO9780511613104},
      review={\MR{1381692 (97g:17016)}},
}

\bib{MR1016381}{book}{
      author={Manin, Yu.~I.},
       title={Quantum groups and noncommutative geometry},
   publisher={Universit{\'e} de Montr{\'e}al Centre de Recherches
  Math{\'e}matiques},
     address={Montreal, QC},
        date={1988},
        ISBN={2-921120-00-3},
      review={\MR{1016381 (91e:17001)}},
}

\bib{MR1243637}{book}{
      author={Montgomery, Susan},
       title={Hopf algebras and their actions on rings},
      series={CBMS Regional Conference Series in Mathematics},
   publisher={Published for the Conference Board of the Mathematical Sciences,
  Washington, DC},
        date={1993},
      volume={82},
        ISBN={0-8218-0738-2},
      review={\MR{1243637 (94i:16019)}},
}

\bib{MR953821}{article}{
      author={Rosso, Marc},
       title={Finite-dimensional representations of the quantum analog of the
  enveloping algebra of a complex simple {L}ie algebra},
        date={1988},
        ISSN={0010-3616},
     journal={Comm. Math. Phys.},
      volume={117},
      number={4},
       pages={581\ndash 593},
         url={http://projecteuclid.org/getRecord?id=euclid.cmp/1104161818},
      review={\MR{953821 (90c:17019)}},
}

\end{biblist}
\end{bibdiv}

\end{document}